\documentclass[11pt,reqno,a4paper]{amsart}
\usepackage[british]{babel}

\usepackage{graphicx}
\usepackage{amscd}
\usepackage{amsmath}
\usepackage{amsfonts}
\usepackage{amssymb}
\usepackage{comment}
\usepackage{color}
\usepackage[usenames,dvipsnames]{xcolor}
\usepackage{pdfsync}
\usepackage{bbm}
\usepackage{dsfont}
\usepackage[colorlinks]{hyperref}

\definecolor{trp}{rgb}{1,1,1}

\definecolor{red}{rgb}{1,0,.2}

\newtheorem{theorem}{Theorem}[section]
\theoremstyle{plain}

\newtheorem{definition}[theorem]{Definition}

\newtheorem{lemma}[theorem]{Lemma}

\newtheorem{prop}[theorem]{Proposition}
\newtheorem{remark}[theorem]{Remark}

\newtheorem*{ass}{Assumption~A}
\numberwithin{equation}{section}

\newcommand{\R}{\mathbb{R}}
\newcommand{\N}{\mathbb{N}}

\newcommand{\proj}{\mathrm{proj}}
\newcommand{\lv}{\boldsymbol\lambda}
\newcommand{\ev}{\boldsymbol\varepsilon}

\newcommand{\vv}{v}
\newcommand{\yv}{y}
\newcommand{\ii}{\mathbf{i}}
\newcommand{\jj}{\mathbf{j}}

\newcommand{\xv}{x}
\newcommand{\tv}{t}

\newcommand{\0}{\boldsymbol{0}}
\newcommand{\1}{\boldsymbol{N\!\!-\!\!1}}
\newcommand{\iiv}{\overline{\imath}}
\newcommand{\jjv}{\overline{\jmath}}
\newcommand{\al}[1]{\langle #1\rangle}

\usepackage[textsize=small]{todonotes}

\begin{document}
\title[Pointwise regularity of parameterized affine zipper fractal curves]{Pointwise regularity of parameterized affine zipper fractal curves}

\author{Bal\'azs B\'ar\'any}
\address{Bal\'azs B\'ar\'any, Budapest University of Technology and Economics, MTA-BME Stochastics Research Group, P.O.Box 91, 1521 Budapest, Hungary\newline
        Mathematics Institute, University of Warwick, Coventry CV4 7AL, UK} \email{balubsheep@gmail.com}

\author{Gergely Kiss}
\address{Gergely Kiss, Budapest University of Technology and Economics, MTA-BME Stochastics Research Group, P.O.Box 91, 1521 Budapest, Hungary \newline
  Faculty of Science, University of Luxembourg, Luxembourg } \email{kigergo57@gmail.com}

\author{Istv\'an Kolossv\'ary}
\address{Istv\'an Kolossv\'ary, Budapest University of Technology and Economics,  Institute of Mathematics, Department of Stochastics;  \newline MTA Alfr\'ed R\'enyi Institute of Mathematics} \email{istvanko@math.bme.hu}


\begin{abstract}
We study the pointwise regularity of zipper fractal curves generated by affine mappings. Under the assumption of dominated splitting of index-1, we calculate the Hausdorff dimension of the level sets of the pointwise H\"older exponent for a subinterval of the spectrum. We give an equivalent characterization for the existence of regular pointwise H\"older exponent for Lebesgue almost every point. In this case, we extend the multifractal analysis to the full spectrum. In particular, we apply our results for de Rham's curve.
\end{abstract}
\date{\today}

\maketitle

\thispagestyle{empty}

\section{Introduction and Statements}

Let us begin by recalling the general definition of fractal curves
from Hutchinson~\cite{Hutchinson81} and Barnsley~\cite{Barnsley}.

\begin{definition}\label{def:zipper}
    A system $\mathcal{S} =\left\{f_0,\dots, f_{N-1}\right\}$
    of contracting mappings of $\R^d$ to itself is called a \textnormal{zipper} with vertices $Z=\left\{z_0,\dots, z_N\right\}$
    and signature $\ev = (\varepsilon_0, \dots, \varepsilon_{N-1})$, $\varepsilon_i\in\left\{0,1\right\}$, if the cross-condition
    $$ f_i(z_0)=z_{i+\varepsilon_i} \text{ and } f_i(z_N) = z_{i+1-\varepsilon_i}$$
    holds for every $i = 0,\dots,N-1$. We call the system a \textnormal{self-affine zipper }if the functions $f_i$ are affine contractive mappings of the form
    \begin{equation*}
    f_i(x) = A_i x +t_i, \text{ for every } i\in\{0,1,\ldots,N-1\},
    \end{equation*}
    where $A_i\in\R^{d\times d}$ invertible and $t_i\in\R^d$.

    The \textnormal{fractal curve} generated from $\mathcal{S}$ is the unique non-empty compact set $\Gamma$, for which
    $$
    \Gamma=\bigcup_{i=0}^{N-1}f_i(\Gamma).
    $$
    If $\mathcal{S}$ is an affine zipper then we call $\Gamma$ a \textnormal{self-affine curve}.
\end{definition}

\begin{figure}
    \centering
    \includegraphics[width=60mm]{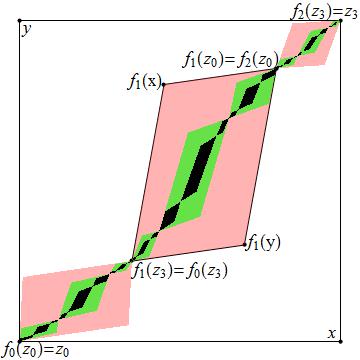}
    \caption{An affine zipper with $N=3$ maps and signature $\ev = (0, 1, 0)$.}\label{fig:zipper}
\end{figure}
For an illustration see Figure \ref{fig:zipper}. It shows the first
({\color{red}red}), second ({\color{green}green}) and third (black)
level cylinders of the image of $[0, 1]^2$. The cross-condition
ensures that $\Gamma$ is a continuous curve.

The dimension theory of self-affine curves is far from being well
understood. The Hausdorff dimension of such curves is known only in
a very few cases. The usual techniques, like self-affine
transversality, see Falconer \cite{FalconerSA}, Jordan, Pollicott
and Simon \cite{JordanPollicottSimon}, destroys the curve structure.
Ledrappier~\cite{Ledrappier} gave a sufficient condition to
calculate the Hausdorff dimension of some fractal curves, and
Solomyak~\cite{Solomyak} applied it to calculate the dimension of
the graph of the Takagi function for typical parameters. Feng and
K{\"a}enm{\"a}ki~\cite{feng2016self} characterized self-affine
systems, which has analytic curve attractor. Let us denote the
$s$-dimensional Hausdorff measure and the Hausdorff dimension of a
set $A$ by $\mathcal{H}^s(A)$ and $\dim_HA$, respectively. Moreover,
let us denote the Packing and (upper) box-counting dimension of a
set $A$ by $\dim_PA$ and $\overline{\dim}_BA$, respectively. For
basic properties and definition of Hausdorff-, Packing- and
box-counting dimension, we refer to \cite{Falconer}.

Bandt and Kravchenko~\cite{BandtKrav} studied some smoothness
properties of self-affine curves, especially the tangent lines of
planar self-affine curves. The main purpose of this paper is to
analyse the pointwise regularity of affine curves under some
parametrization. Let us recall the definition of pointwise H\"older
exponent of a real valued function $g$, see for example \cite[eq.
(1.1)]{Jaffard1997I}. We say that $g\in C^{\beta}(x)$ if there exist
a $\delta>0$, $C>0$ and a polynomial $P$ with degree at most
$\lfloor\beta\rfloor$ such that
$$
|g(y)-P(y-x)|\leq C|x-y|^{\beta}\text{ for every $y\in B_{\delta}(x)$,}
$$
where $B_{\delta}(x)$ denotes the ball with radius $\delta$ centered
at $x$. Let $\alpha_p(x)=\sup\{\beta:g\in C^{\beta}(x)\}$. We call
$\alpha_p(x)$ the {\it pointwise H\"older exponent }of $g$ at the
point $x$.

We call $F\colon\R^m\mapsto\R$ a {\it self-similar function} if
there exists a bounded open set $U\subset\R^m$, and
 contracting similarities $g_1,\ldots,g_k$ of $\R^m$ such that
$g_i(U)\cap g_j(U)=\emptyset$ and $g_i(U)\subset U$ for every $i\neq
j$, and a smooth function $g\colon\R^m\mapsto\R$, and real numbers
$|\lambda_i|<1$ for $i=1,\ldots,k$ such that
\begin{equation}\label{eq:selfsim}
F(x)=\sum_{i=1}^k\lambda_iF(g_i^{-1}(x))+g(x),
\end{equation}
see \cite[Definition~2.1]{Jaffard1997II}. The multifractal formalism
of the pointwise H\"older exponent of self-similar functions was
studied in several aspects, see for example Aouidi and Slimane
\cite{Aouidi2004}, Slimane
\cite{Slimane2001,Slimane2003,Slimane2012} and Saka \cite{Saka2005}.

Hutchinson \cite{Hutchinson81} showed that the family of contracting
functions $g_1,\ldots,g_k$ has a unique, non-empty compact invariant
set $\Omega$ (called the {\it attractor of
$\Phi=\{g_1,\ldots,g_k\}$}), i.e.
$\Omega=\bigcup_{i=1}^kg_i(\Omega)$. We note that in the case of
self-similar function, the graph of $F$ (denoted by
$\mathrm{Graph}(F)$) over the set $\Omega$ can be written as the
unique, non-empty, compact invariant set of the family of functions
$S_1,\ldots,S_k$ in $\R^{m+1}$, where
$$
S_i(x,y)=(g_i(x),\lambda_iy+g(g_i(x))).
$$

In this paper, we study the local regularity of a generalized
version of self-similar functions. Namely, let
$\lv=(\lambda_0,\dots,\lambda_{N-1})$ be a probability vector. Let
us subdivide the interval $[0,1]$ according to the probability
vector $\lv$ and signature $\ev= (\varepsilon_0, \dots,
\varepsilon_{N-1})$, $\varepsilon_i\in\left\{0,1\right\}$ of the
zipper $\mathcal{S}$. Let $g_i$ be the linear function mapping the
unit interval [0,1] to the $i$th subinterval of the division which
is order-preserving or order-reversing according to the signature
$\varepsilon_i$. That is, the interval $[0,1]$ is the attractor of
the iterated function system
\begin{equation}\label{eq:paramIFS}
\Phi=\left\{g_i:x\mapsto (-1)^{\varepsilon_i}\lambda_ix+\gamma_i\right\}_{i=0}^{N-1},
\end{equation}
where $\gamma_i=\sum_{j=0}^{i-1}\lambda_j+\varepsilon_i\lambda_i$.
Let $v:[0,1]\mapsto\Gamma\subset\R^d$ be the unique continuous
function satisfying the functional equation
\begin{equation}\label{eq:linparam}
v(x)=f_i\left(v(g_i^{-1}(x))\right)\text{ if }x\in g_i([0,1]).
\end{equation}
We note that
$g_i^{-1}(x)=\frac{x-\gamma_i}{(-1)^{\varepsilon_i}\lambda_i}$, and
$g_i^{-1}(x)\in[0,1]$ if and only if $x\in g_i([0,1])$. Thus,
$\mathrm{Graph}(v)$ is the attractor of the IFS
 $$
 S_i(x,y)=(g_i(x),A_iy+t_i).
 $$
 We call $v$ as the \textit{linear parametrization of $\Gamma$}.
 Such linear parameterizations occur in the study of Wavelet functions in a natural way, see for example Protasov~\cite{Protasov06}, Protasov and Guglielmi \cite{PG}, and Seuret \cite{Seuret2016}. A particular example for \eqref{eq:linparam} is the de Rham's curve, see Section~\ref{Sec:Ex_deRham} for details including an example of a graph of $v$ generated by the de Rham's curve.

The main difference between the self-similar function $F$ defined in
\eqref{eq:selfsim} and $v$ defined in \eqref{eq:linparam} is the
contraction part. Namely, while $F$ is a real valued function
rescaled by only a real number, the function $v$ is $\R^d$ valued
and a strict affine transformation is acting on it. This makes the
study of such functions more difficult. As a slight abuse of the
appellation of the pointwise H\"older exponent, we redefine the
pointwise H\"older exponent $\alpha(x)$ of the function $v$ at a
point $x\in[0,1]$ as
\begin{equation}\label{eq:alpha1}
\alpha(x)=\liminf_{y\to x}\frac{\log\|v(x)-v(y)\|}{\log|x-y|}.
\end{equation}
We note that if $\alpha_p(x)<1$ or $\alpha(x)<1$ then
$\alpha_p(x)=\alpha(x)$. Otherwise, we have only
$\alpha(x)\leq\alpha_p(x)$.

When the $\liminf$ in \eqref{eq:alpha1} exists as a limit, then we
say that $v$ has a regular pointwise H\"older exponent $\alpha_r(x)$
at a point $x\in[0,1]$, i.e.
\begin{equation}
\alpha_r(x)=\lim_{y\to x}\frac{\log\|v(x)-v(y)\|}{\log|x-y|}. \label{eq:alpha2}
\end{equation}

Let us define the level sets of the (regular) pointwise H\"older exponent by
\begin{align*}
E(\beta)&=\left\{x\in[0,1]:\alpha(x)=\beta\right\}\text{ and} \\
E_r(\beta)&=\left\{x\in[0,1]:\alpha_r(x)=\beta\right\}.
\end{align*}

Our goal is to perform multifractal analysis, i.e. to study the
possible values, which occur as (regular) pointwise H\"older
exponents, and determine the magnitude of the sets, where it
appears. This property was studied for several types of singular
functions, for example for wavelets by  Barral and Seuret
\cite{BarralSeuret}, Seuret \cite{Seuret2016},  for Weierstrass-type
functions Otani \cite{Otani}, for complex analogues of the Takagi
function by Jaerisch and Sumi \cite{Sumi} or for different
functional equations by Coiffard, Melot and Willer \cite{CMW}, by
Okamura \cite{Okamura} and by Slimane \cite{Slimane2003} etc.

The main difficulty in our approach is to handle the distance
$\|v(x)-v(y)\|$. In the previous examples, the function $F$ defined
with the equation \eqref{eq:selfsim} was scaled only by a constant.
Roughly speaking
$$
\|F(g_{i_1,\ldots,i_n}(x))-F(g_{i_1,\ldots,i_n}(y))\|\approx |\lambda_{i_1}\cdots\lambda_{i_n}|\|F(x)-F(y)\|.
$$
In the case of self-affine systems, this is not true anymore. That is,
$$
\|v(g_{i_1,\ldots,i_n}(x))-v(g_{i_1,\ldots,i_n}(y))\|\approx\|A_{i_1}\cdots A_{i_n}(v(x)-v(y))\|.
$$
However, in general $\|A_{i_1}\cdots
A_{i_n}(v(x)-v(y))\|\not\approx\|A_{i_1}\cdots
A_{i_n}\|\|v(x)-v(y)\|$. In order to be able to compare the distance
$\|v(g_{i_1,\ldots,i_n}(x))-v(g_{i_1,\ldots,i_n}(y))\|$ with the
norm of the product of matrices, we need an extra assumption on the
family of matrices.

Let us denote by $M^o$ the interior and by $\overline{M}$ the
closure of a set $M\subseteq \mathbb{PR}^{d-1}$. For a point
$v\in\R^d$, denote $\langle v\rangle$ equivalence class of $v$ in
the projective space $\mathbb{PR}^{d-1}$. Every invertible matrix
$A$ defines a natural map on the projective space
$\mathbb{PR}^{d-1}$ by $\al{v}\mapsto\al{Av}$. As a slight abuse of
notation, we denote this function by $A$ too.

\begin{definition}\label{def:domsplit}
    We say that a family of matrices $\mathcal{A}=\left\{A_0,\dots,A_{N-1}\right\}$ have \textnormal{dominated splitting of index-1} if there exists a non-empty open subset $M\subset\mathbb{PR}^{d-1}$ with a finite number of connected components with pairwise disjoint closure such that $$\bigcup_{i=0}^{N-1}A_i\overline{M}\subset M^{o},$$
    and there is a $d-1$ dimensional hyperplane that is transverse to all elements of $\overline{M}$. We call the set $M$ a multicone.
\end{definition}

We adapted the definition of dominated splitting of index-1 from the
paper of Bochi and Gourmelon \cite{BochiGour}. They showed that the
tuple of matrices $\mathcal{A}$ satisfies the property in
Definition~\ref{def:domsplit} if and only if there exist constants
$C>0$ and $0<\tau<1$ such that
$$
\frac{\alpha_2(A_{i_1}\cdots A_{i_{n}})}{\alpha_1(A_{i_1}\cdots A_{i_{n}})}\leq C\tau^n
$$
for every $n\geq1$ and $i_0,\ldots,i_{n-1}\in\{0,\ldots,N-1\}$,
where $\alpha_i(A)$ denotes the $i$th largest singular value of the
matrix $A$. That is, the weakest contracting direction and the
stronger contracting directions are strongly separated away
(splitted), $\alpha_1$ dominates $\alpha_2$. This condition makes it
easier to handle the growth rate of the norm of matrix products,
which will be essential in our later studies.

We note that for example a tuple $\mathcal{A}$ formed by matrices
with strictly positive elements, satisfies the dominated splitting
of index-1 of $M=\langle\{x\in\R^d:x_i>0,i=1,\ldots,d\}\rangle$.
Throughout the paper we work with affine zippers, where we assume
that the matrices $A_i$ have dominated splitting of index-1. For
more details, see Section~\ref{sec:ds} and \cite{BochiGour}.

For a subset $M$ of $\mathbb{PR}^{d-1}$ and a point $\xv\in\R^d$,
let $$M(\xv)=\{y\in\R^d:\langle\yv-\xv\rangle\in M\}.$$ We call the
set $M(\xv)$ a cone centered at $\xv$.

We say that $\mathcal{S}$ satisfies the \textit{strong open set
condition} (SOSC), if there exists an open bounded set $U$ such that
$$
f_i(U)\cap f_j(U)=\emptyset\text{ for every }i\neq j\text{ and }\Gamma\cap U\neq\emptyset.
$$
We call $\mathcal{S}$ a \textit{non-degenerate system}, if it
satisfies the SOSC and $\langle z_N-z_0\rangle\notin
\bigcap_{k=0}^{\infty}\bigcap_{|\iiv|=k} A_{\iiv}^{-1}(M^c)$, where
for a finite length word $\iiv=i_1\ldots i_k$, $A_{\iiv}$ denotes
the matrix product $A_{i_1}A_{i_2}\ldots A_{i_k}$. We note that the
non-degenerate condition guarantees that the curve
$v\colon[0,1]\mapsto\R^d$ is not self-intersecting and it is not
contained in a strict hyperplane of $\R^d$.

Denote by $P(t)$ the pressure function which is defined as the unique root of the equation
\begin{equation}\label{eq:pressuredef1}
0= \lim_{n\to \infty}\frac{1}{n}\log \sum_{i_1,\dots,i_n=0}^{N-1}\|A_{i_1}\cdots A_{i_n}\|^t\left(\lambda_{i_1}\cdots\lambda_{i_n}\right)^{-P(t)}.
\end{equation}

A considerable attention has been paid for pressures, which are
defined by matrix norms, see for example K\"aenm\"aki
\cite{Kaenmaki}, Feng and Shmerkin \cite{FengShmerkin}, and Morris
\cite{Morris1, Morris2}. Feng \cite{FengLyap1-2003} and later Feng
and Lau \cite{FengLau_pressure} studied the properties of the
pressure $P$ for positive and non-negative matrices. In
Section~\ref{sec:ds}, we extend these results for the dominated
splitting of index-1 case. Namely, we will show that the function
$P:\R\mapsto\R$ is continuous, concave, monotone increasing, and
continuously differentiable.

Unfortunately, even for positive matrices, the computation of the
precise values of $P(t)$ is hopeless. For a fast approximation
algorithm, see Pollicott and Vytnova~\cite{pollicott2015estimating}.

Let $d_0>0$ be the unique real number such that
\begin{equation}\label{eq:d0}
0= \lim_{n\to \infty}\frac{1}{n}\log \sum_{|\iiv|=n}\|A_{\iiv}\|^{d_0}.
\end{equation}
Observe that for every $n\geq1$, $\{f_{\iiv}(U):|\iiv|=n\}$ defines
a cover of $\Gamma$. But since $\Gamma$ is a curve and thus
$\dim_H\Gamma\geq1$, and since every $f_{\iiv}(U)$ can be covered by
a ball with radius $\|A_{\iiv}\||U|$, $d_0\geq1$.

Let $$\alpha_{\min}=\lim_{t\to+\infty}\frac{P(t)}{t},\ \alpha_{\max}=\lim_{t\to-\infty}\frac{P(t)}{t}\text{ and }\widehat{\alpha}=P'(0).$$

The values $\alpha_{\min}$ and $\alpha_{\max}$ correspond to the
logarithm of the joint- and the lower-spectral radius defined by
Protasov \cite{Protasov06}.

We say that the function $v:[0,1]\mapsto\Gamma$ is \textit{symmetric} if
\begin{equation}\label{esymm}
\lambda_0=\lambda_{N-1}\text{ and }\lim_{k\to\infty}\frac{\|A_0^{k}\|}{\|A_{N-1}^k\|}=1.
\end{equation}

Now, we state our main theorems on the pointwise H\"older exponents.

\begin{theorem}\label{thm:nondeg}
    Let $\mathcal{S}$ be a non-degenerate system. Then there exists a constant $\widehat{\alpha}$ such that for $\mathcal{L}$-a.e. $x\in[0,1]$, $\alpha(x)=\widehat{\alpha}\geq1/d_0$. Moreover, there exists an $\varepsilon>0$ such that for every $\beta\in[\widehat{\alpha},\widehat{\alpha}+\varepsilon]$
    \begin{equation}\label{eq:nondegmulti}
    \dim_H\left\{x\in[0,1]:\alpha(x)=\beta\right\}=\inf_{t\in\R}\{t\beta-P(t)\}.
    \end{equation}

If $\mathcal{S}$ satisfies \eqref{esymm} then \eqref{eq:nondegmulti} can be extended for every $\beta\in[\alpha_{\min},\widehat{\alpha}+\varepsilon]$.

Furthermore, the functions $\beta\mapsto \dim_H E(\beta)$ and
$\beta\mapsto \dim_H E_r(\beta)$ are continuous and concave on their
respective domains.
\end{theorem}

In the following, we give a sufficient condition to extend the
previous result, where \eqref{eq:nondegmulti} holds to the complete
spectrum $[\alpha_{\min},\alpha_{\max}]$. As a slight abuse of
notation for every $\theta\in\mathbb{PR}^{d-1}$, we say that
$\underline{0}\neq v\in\theta$ if $\al{v}=\theta$.

\begin{ass}
    For a nondegenerate affine zipper $\mathcal{S}=\left\{f_i:\xv\mapsto A_i\xv+\tv_i\right\}_{i=0}^{N-1}$ with vertices $\left\{z_0,\dots,z_N\right\}$ assume that there exists a convex, simply connected closed cone $C\subset\mathbb{PR}^{d-1}$ such that
    \begin{enumerate}
        \item $\bigcup_{i=1}^NA_iC\subset C^o$ and for every $\underline{0}\neq v\in\theta\in C$, $\langle A_i\vv,\vv\rangle>0$,
        \item $\langle z_N-z_0\rangle\in C^o$.
    \end{enumerate}
\end{ass}

Observe that if $\mathcal{S}$ satisfies Assumption A then it
satisfies the strong open set condition with respect to the set $U$,
which is the bounded component of $C^o(z_0)\cap C^o(z_N)$. We note
that if all the matrices have strictly positive elements and the
zipper has signature $(0,\dots,0)$ then Assumption~A holds.

\begin{theorem}\label{thm:assA}
    Let $\mathcal{S}$ be an affine zipper satisfying Assumption A.  Then for every $\beta\in[\widehat{\alpha},\alpha_{\max}]$
    \begin{equation}\label{eq:assAmulti}
    \dim_H\left\{x\in[0,1]:\alpha(x)=\beta\right\}=\inf_{t\in\R}\{t\beta-P(t)\},
    \end{equation}
and for every $\beta\in[\alpha_{\min},\alpha_{\max}]$
\begin{equation}\label{eq:assAmulti2}
\dim_H\left\{x\in[0,1]:\alpha_r(x)=\beta\right\}=\inf_{t\in\R}\{t\beta-P(t)\}.
\end{equation}
Moreover, if $\mathcal{S}$ satisfies \eqref{esymm} then \eqref{eq:assAmulti} can be extended for every $\beta\in[\alpha_{\min},\alpha_{\max}]$.

The functions $\beta\mapsto \dim_H E(\beta)$ and
    $\beta\mapsto \dim_H E_r(\beta)$ are continuous and concave on their
    respective domains.
\end{theorem}

Assumption A has another important role. In Theorem~\ref{thm:assA},
we calculated the spectrum for the regular H\"older exponent,
providing that it exists. We show that the existence of the regular
H\"older exponent for Lebesgue typical points is equivalent to
Assumption A.

\begin{theorem}\label{thm:equiv}
Let $\mathcal{S}$ be a non degenerate system. Then the regular
H\"older exponent exists for Lebesgue almost every point if and only
if $\mathcal{S}$ satisfies Assumption A. In particular,
$\alpha_r(x)=P'(0)$ for Lebesgue almost every $x\in[0,1]$.
\end{theorem}

\begin{remark}
    In the sequel, to keep the notation tractable we assume the signature $\ev = (0, \dots,0)$. The results
    carry over for general signatures, and the proofs can be easily modified for the general signature case, see Remark~\ref{rem:final}.
\end{remark}

The organization of the paper is as follows. In Section~\ref{sec:ds}
we prove several properties of the pressure function $P(t)$,
extending the works of \cite{FengLyap1-2003,FengLau_pressure} to the
dominated splitting of index-1 case using \cite{BochiGour}. We prove
Theorem \ref{thm:nondeg} in Section \ref{sec:nondegZippers}. Section
\ref{sec:AssAZippers} contains the proofs of Theorems \ref{thm:assA}
and \ref{thm:equiv} when the zipper satisfies Assumption A. Finally,
as an application in Section \ref{Sec:Ex_deRham}, we show that our
results can be applied to de Rham's curve, giving finer results than
existing ones in the literature.

\section{Pressure for matrices with dominated splitting of index-1}\label{sec:ds}

In this section, we generalize the result of Feng
\cite{FengLyap1-2003}, and Feng and Lau \cite{FengLau_pressure}. In
\cite{FengLau_pressure} the authors studied the pressure function
and multifractal properties of Lyapunov exponents for products of
positive matrices. Here, we extend their results for a more general
class of matrices by using Bochi and Gourmelon \cite{BochiGour} for
later usage.

Let $\Sigma$ be the set of one side infinite length words of symbols $\left\{0,\dots,N-1\right\}$, i.e. $\Sigma=\{0,\ldots,N-1\}^{\mathds{N}}$. Let $\sigma$ denote the left shift on $\Sigma$, its $n$-fold composition by $\sigma^n\ii=(i_{n+1},i_{n+2},\ldots)$. We use the standard notation $\ii|n$ for $i_1,\ldots,i_n$ and
$$
[\ii|_n]:=\left\{\jj\in\Sigma:j_1=i_1,\dots,j_n=i_n\right\}.
$$

Let us denote the set of finite length words by
$\Sigma^*=\bigcup_{n=0}^{\infty}\{0,\ldots,N-1\}^n$, and for an
$\iiv\in\Sigma^*$, let us denote the length of $\iiv$ by $|\iiv|$.
For a finite word $\iiv\in\Sigma^*$ and for a $\jj\in\Sigma$, denote
$\iiv\jj$ the concatenation of the finite word $\iiv$ with $\jj$.

Denote $\ii\wedge\jj$ the length of the longest common prefix of $\ii,\jj\in\Sigma$, i.e. $\ii\wedge\jj = \min\{n-1:\,i_n\neq j_n\}$. Let $\lv=(\lambda_0,\dots,\lambda_{N-1})$ be a probability vector and let $d(\ii,\jj)$ be the distance on $\Sigma$ with respect to $\lv$. Namely,
$$
d(\ii,\jj)=\prod_{n=1}^{\ii\wedge\jj}\lambda_{i_n}=:\lambda_{\ii|_{\ii\wedge\jj}}.
$$
If $\ii\wedge\jj=0$ then by definition $\ii|_{\ii\wedge\jj}=\emptyset$ and $\lambda_{\emptyset}=1$.
For every $r>0$, we define a partition $\Xi_r$ of $\Sigma$ by
\begin{equation}\label{eq:xi}
\Xi_r=\left\{[i_1,\dots,i_n]:\lambda_{i_1}\cdots\lambda_{i_n}\leq
r<\lambda_{i_1}\cdots\lambda_{i_{n-1}}\right\}.
\end{equation}

For a matrix $A$ and a subspace $\theta$, denote $\|A|\theta\|$ the
norm of $A$ restricted to $\theta$, i.e. $\|A|\theta\|=\sup_{v\in
\theta}\|Av\|/\|v\|$. In particular, if $\theta$ has dimension one
$\|A|\theta\|=\|Av\|/\|v\|$ for any $\underline{0}\neq v\in\theta$.
Denote $G(d,k)$ the Grassmanian manifold of $k$ dimensional
subspaces of $\mathbb{R}^d$. We define the angle between a $1$
dimensional subspace $E$ and a $d-1$ dimensional subspace $F$ as
usual, i.e.
$$
\sphericalangle(E,F)=\arccos\left(\frac{\langle v,\mathrm{proj}_Fv\rangle}{\|\mathrm{proj}_Fv\|\|v\|}\right),
$$
where $\underline{0}\neq v\in E$ arbitrary and $\mathrm{proj}_F$ denotes the orthogonal projection onto $F$.

The following theorem collects the most relevant properties of a family of matrices with dominated splitting of index-1.

\begin{theorem}\cite[Theorem~A, Theorem~B, Claim on p. 228]{BochiGour}\label{thm:BochiGour}
    Suppose that a finite set of matrices $\{A_0,\dots,A_{N-1}\}$ satisfies the dominated splitting of index-1 with multicone $M$. Then there exist H\"older continuous functions $E:\Sigma\mapsto\mathbb{PR}^{d-1}$ and $F:\Sigma\mapsto G(d,d-1)$ such that
    \begin{enumerate}
        \item $E(\ii)=A_{i_1}E(\sigma\ii)$ for every $\ii\in\Sigma$,
        \item $F(\ii)=A_{i_1}^{-1}F(\sigma\ii)$ for every $\ii\in\Sigma$,\label{item:BochiGour2}
        \item there exists $\beta>0$ such that $\sphericalangle(E(\ii),F(\jj))>\beta$ for every $\ii,\jj\in\Sigma$,
        \item\label{item:bg4} there exist constants $C\geq1$ and $0<\tau<1$ such that
        $$
        \frac{\alpha_2(A_{\ii|_n})}{\|A_{\ii|_n}\|}\leq C\tau^n
        $$
        for every $\ii\in\Sigma$ and $n\geq1$,
        \item\label{} there exists a constant $C>0$ such that $\|A_{\ii|_n}|E(\sigma^n\ii)\|\geq C\|A_{\ii|_n}\|$ for every $\ii\in\Sigma$,\label{item:bochigour5}
        \item there exists a constant $C>0$ such that $\|A_{\ii|_n}|F(i_n\dots i_1\jj)\|\leq C\alpha_2(A_{\ii|_n})$ for every $\ii,\jj\in\Sigma$.\label{item:bg6}
    \end{enumerate}
\end{theorem}

In particular, if $M$ is the multicone from
Definition~\ref{def:domsplit}, then
$$
E(\ii)=\bigcap_{n=1}^{\infty}A_{i_1}\cdots A_{i_n}(M),
$$
and for every $V\in M$, $A_{i_1}\cdots A_{i_n}V\to E(\ii)$ uniformly
(independently of $V$). Hence, there exists a constant $C>0$ such
that for every $V\in M$ and every $\iiv\in\Sigma^*$,
\begin{equation}\label{eq:compare}
\|A_{\iiv}|V\|\geq C'\|A_{\iiv}\|.
\end{equation}
So, this gives us a strong control over the growth rate of matrix products on subspaces in $M$.

Another consequence of Theorem~\ref{thm:BochiGour} is that the
function $\psi(\ii)=\log\|A_{i_1}|E(\sigma\ii)\|$ is
H\"older-continuous. That is, there exist $C>0$ and $0<\tau<1$ such
that
\begin{equation}\label{eq:hc}
|\psi(\ii)-\psi(\jj)|\leq C\tau^{\ii\wedge\jj}.
\end{equation}
Moreover, by the property $E(\ii)=A_{i_1}E(\sigma\ii)$, we have
\begin{equation}\label{eq:multip}
\|A_{\ii|_n}|E(\sigma^n\ii)\|=\prod_{k=1}^n\|A_{i_k}|E(\sigma^k\ii)\|.
\end{equation}
Indeed, since $E(\ii)$ is a one dimensional subspace, for every $v\in E(\sigma^n\ii)$
$$
\|A_{\ii|_n}|E(\sigma^n\ii)\|=\frac{\|A_{\ii|_n}v\|}{\|v\|}=\prod_{k=1}^n\frac{\|A_{i_k\ldots i_n}v\|}{\|A_{i_{k+1},\ldots,i_n}v\|}=\prod_{k=1}^n\|A_{i_k}|E(\sigma^k\ii)\|.
$$

\begin{remark}
    We note if the multicone $M$ in Definition~\ref{def:domsplit} has only one connected component then it can be chosen to be simply connected and convex. Indeed, since $\overline{M}$ is separated away from the strong stable subspaces $F$ then $\mathrm{cv}(\overline{M})$ must be separated away from every $d-1$ dimensional strong stable subspace, as well, where $\mathrm{cv}(\overline{M})$ denotes the convex hull of $\overline{M}$. Thus $A_i(\mathrm{cv}(\overline{M}))\subset\mathrm{cv}(\overline{M})^o$ for every $i$.
\end{remark}
For every $t$, let $\varphi_t:\Sigma\mapsto\R$ be the potential
function defined by
\begin{equation}\label{eq:pot}
\varphi_t(\ii)=\log\left(\|A_{i_1}|E(\sigma\ii)\|^t\lambda_{i_1}^{-P(t)}\right),
\end{equation}
where $P(t)$ was defined in \eqref{eq:pressuredef1}.

Using Theorem~\ref{thm:BochiGour}, one can show that for every $t$,
$\varphi_t$ is a H\"older continuous function. Thus, by
\cite[Theorem~1.4]{Bowenbook}, for every $t\in\R$ there exists a
unique $\sigma$-invariant, ergodic probability measure $\mu_t$ on
$\Sigma$ such that there exists a constant $C(t)>1$ such that for
every $\ii\in\Sigma$ and every $n\geq1$
\begin{equation}\label{eq:Gibbs}
C(t)^{-1}\leq\frac{\mu_t([\ii|_n])}{\prod_{k=0}^{n-1}e^{\varphi_t(\sigma^k\ii)}}\leq C(t).
\end{equation}
Observe that
$$
\prod_{k=0}^{n-1}e^{\varphi_t(\sigma^k\ii)}=\|A_{\ii|_n}|E(\sigma^n\ii)\|^t\cdot\lambda_{\ii|_n}^{-P(t)}.
$$

Moreover,
\begin{equation}\label{eq:dimhmut}
\dim_H\mu_t=\frac{h_{\mu_t}}{\chi_{\mu_t}},
\end{equation}
where
\begin{eqnarray}
h_{\mu_t} & = & \lim_{n\to\infty}\frac{-1}{n}\sum_{|\iiv|=n}\mu_t([\iiv])\log\mu_t([\iiv])=-\int\varphi_t(\ii)d\mu_t(\ii),\label{eq:Gibbsent}\\
\chi_{\mu_t} & = &
\lim_{n\to\infty}\frac{-1}{n}\sum_{|\iiv|=n}\mu_t([\iiv])\log\lambda_{\iiv}=-\int\log\lambda_{i_1}d\mu_t(\ii).\label{eq:GibbsLyap}
\end{eqnarray}

We call $\chi_{\mu_t}$ the Lyapunov exponent of $\mu_t$ and
$h_{\mu_t}$ the entropy of $\mu_t$.

\begin{lemma}\label{lem:pressure}
    The map $t\mapsto P(t)$ is continuous, concave, monotone increasing on $\R$.
\end{lemma}

\begin{proof}
    Since $\mu_t$ is a probability measure on $\Sigma$ and $\Xi_r$ is a partition we get
    $$
    0=\frac{\log\sum_{\iiv\in\Xi_r}\mu_t([\iiv])}{\log r}\text{ for every $r>0$}
    $$
    and by \eqref{eq:Gibbs} and \eqref{eq:pressuredef1}
    \begin{equation}\label{eq:press2}
    P(t)=\lim_{r\to0+}\frac{\log\sum_{\iiv\in\Xi_r}\|A_{\iiv}\|^t}{\log r}.
    \end{equation}
    Using this form it can be easily seen that $t\mapsto P(t)$ is continuous, concave and monotone increasing.
\end{proof}

By Lemma~\ref{lem:pressure}, the potential $\varphi_t$ depends
continuously on $t$. Moreover, by \eqref{eq:hc},
$|\varphi(\ii)-\varphi(\jj)|\leq Ct\tau^{\ii\wedge\jj}$. Thus, the
Perron-Frobenius operator
 $$
 (T_t(g))(\ii)=\sum_{i=0}^{N-1}e^{\varphi_t(i\ii)}g(i\ii)
 $$
 depends continuously on $t$. Hence, the unique eigenfunction $h_t$ of $T_t$ and the eigenmeasure of $\nu_t$ of the dual operator $T_t^*$ depends continuously on $t$. Since $d\mu_t=h_td\nu_t$, see \cite[Theorem~1.16]{Bowenbook}, we got that $t\mapsto\mu_t$ is continuous in weak*-topology.
 Hence, by \eqref{eq:Gibbsent} and \eqref{eq:GibbsLyap}, $t\mapsto h_{\mu_t}$ and $t\mapsto\chi_{\mu_t}$ are continuous on $\R$.

\begin{prop}\label{prop:pressure}
    The map $t\mapsto P(t)$ is continuously differentiable on $\R$. Moreover, for every $t\in\R$
    \[
    \dim_H\mu_t=tP'(t)-P(t),
    \]
    and
    \[
    \lim_{n\rightarrow\infty}\frac{\log\|A_{i_1}\cdots A_{i_n}\|}{\log\lambda_{i_1}\cdots\lambda_{i_n}}=P'(t)\text{ for $\mu_t$-almost every }\ii\in\Sigma.
    \]
\end{prop}

\begin{proof}
    We recall \cite[Theorem~2.1]{H}. That is, since $\mu_t$ is a Gibbs measure
    $$
    \tau_{\mu_t}(q)=\lim_{r\to0+}\frac{\log\sum_{\iiv\in\Xi_r}\mu_t([\iiv])^q}{\log r}
    $$
    is differentiable at $q=1$ and $\tau_{\mu_t}'(1)=\dim_H\mu_t$. On the other hand, by \eqref{eq:Gibbs} and \eqref{eq:press2}
    $$
    \tau_{\mu_t}(q)=P(tq)-P(t)q.
    $$
    Hence, by taking the derivative at $q=1$ we get that $P(t)$ is differentiable for every $t\in\R/\left\{0\right\}$ and
    $$
    \dim_H\mu_t=tP'(t)-P(t).
    $$
    Let us observe that
    by \eqref{eq:pot}, \eqref{eq:dimhmut} and \eqref{eq:Gibbsent}
    $$
    \dim_H\mu_t=t\dfrac{-\int\log\|A_{i_1}|E(\sigma\ii)\|d\mu_t(\ii)}{-\int\log\lambda_{i_1}d\mu_t(\ii)}-P(t).
    $$
    Thus,
    $$
    P'(t)=\dfrac{-\int\log\|A_{i_1}|E(\sigma\ii)\|d\mu_t(\ii)}{-\int\log\lambda_{i_1}d\mu_t(\ii)}\text{ for every }t\neq0.
    $$
    Since $t\mapsto\mu_t$ is continuous in weak*-topology we get that $t\mapsto P'(t)$ is continuous on $\R/\left\{0\right\}$. On the other hand, the left and right hand side limits of $P'(t)$ at $t=0$ exist and are equal. Thus, $t\mapsto P(t)$ is continuously differentiable on $\R$.

    By Theorem~\ref{thm:BochiGour} and ergodicity of $\mu_t$ we get the last assertion of the proposition.
\end{proof}

Let us observe that by the definition of pressure function
\eqref{eq:pressuredef1}, $P(0)=-1$ and thus, $\mu_0$ corresponds to
the Bernoulli measure on $\Sigma$ with probabilities
$(\lambda_0,\ldots,\lambda_{N-1})$. That is,
    $$
    \mu_0([i_1,\ldots,i_n])=\lambda_{i_1}\cdots\lambda_{i_n}.
    $$

\begin{lemma}\label{lem:boundsonder}
    For every finite set of matrices $\mathcal{A}$ with dominated splitting of index-1,  $P'(0)\geq1/d_0$, $P'(d_0)\leq1/d_0$.
    Moreover, $P'(0)>1/d_0$ if and only if $P'(d_0)<1/d_0$ if and only if $\mu_{d_0}\neq\mu_0$.
\end{lemma}

\begin{proof}
By the definition of $P(t)$, \eqref{eq:pressuredef1}, $P(d_0)= 0$,
where $d_0$ is defined in \eqref{eq:d0}. Together with $P(0)=-1$ and
the concavity and differentiability of $P(t)$ (by Lemma
\ref{lem:pressure} and Proposition~\ref{prop:pressure}), we get
$P'(0)\geq1/d_0$, $P'(d_0)\leq1/d_0$. Moreover, $P'(0)>1/d_0$ if and
only if $P'(d_0)<1/d_0$.

On the other hand, by Proposition~\ref{prop:pressure}
$$
\dim_H\mu_{d_0}=d_0P'(d_0)=\lim_{n\to\infty}\frac{\log\|A_{\ii|_n}\|^{d_0}}{\log\lambda_{\ii|_n}}=\frac{h_{\mu_{d_0}}}{\chi_{\mu_{d_0}}}\text{ for $\mu_{d_0}$-a.e. $\ii$,}
$$
where in the last equation we used the definition of $\mu_{d_0}$, the entropy and the Lyapunov exponent. Since $\dim_H\mu_0=1$, if $P'(d_0)<1/d_0$ then $\mu_0\neq\mu_{d_0}$. On the other hand, by \cite[Theorem~1.22]{Bowenbook}, for every $\sigma$-invariant, ergodic measure $\nu$ on $\Sigma$,
$$
\frac{h_{\nu}}{-\int\log\lambda_{i_0}d\nu(\ii)}\leq1\text{ and }\frac{h_{\nu}}{-\int\log\lambda_{i_0}d\nu(\ii)}=1\text{ if and only if }\nu=\mu_0.
$$
Therefore, if $P'(d_0)=1/d_0$ then $\frac{h_{\mu_{d_0}}}{\chi_{\mu_{d_0}}}=1$ and so $\mu_{d_0}=\mu_0$.
\end{proof}

\begin{lemma}\label{lem:upperbound}
    For every $\alpha\in[\alpha_{\min},\alpha_{\max}]$
    \begin{equation}\label{eq:ub1}
    \dim_H\Big\{\ii\in\Sigma:\liminf_{m\rightarrow\infty}\frac{\log \|A_{\ii|_m}\| }{\log\lambda_{\ii|_m}}\leq\alpha\Big\}\leq\inf_{t\geq 0}\left\{t\alpha-P(t)\right\}
    \end{equation}
    and
    \begin{equation}\label{eq:ub2}
    \dim_H\Big\{\ii\in\Sigma:\limsup_{m\rightarrow\infty}\frac{\log \|A_{\ii|_m}\| }{\log\lambda_{\ii|_m}}\geq\alpha\Big\}\leq\inf_{t\leq 0}\left\{t\alpha-P(t)\right\}
    \end{equation}
\end{lemma}

\begin{proof}
    For simplicity, we use the notations $$
    \underline{G}_{\alpha}=\left\{\ii\in\Sigma:\liminf_{m\rightarrow\infty}\frac{\log \|A_{\ii|_m}\| }{\log\lambda_{\ii|_m}}\leq\alpha\right\}\text{ and }\overline{G}_{\alpha}=\left\{\ii\in\Sigma:\limsup_{m\rightarrow\infty}\frac{\log \|A_{\ii|_m}\| }{\log\lambda_{\ii|_m}}\geq\alpha\right\}.$$

    Let $\varepsilon>0$ be arbitrary but fixed and let us define the following sets of cylinders:
    \[
    \underline{D}_{r}(\varepsilon)=\Big\{[\iiv]\in\Xi_{\rho}:0<\rho\leq r\text{ and }\frac{\log \|A_{\iiv}\| }{\log\lambda_{\iiv}}\leq\alpha+\varepsilon\Big\}
    \]
    and
    \[
    \overline{D}_r(\varepsilon)=\Big\{[\iiv]\in\Xi_{\rho}:0<\rho\leq r\text{ and }\frac{\log \|A_{\iiv}\| }{\log\lambda_{\iiv}}\geq\alpha-\varepsilon\Big\}.
    \]
    By definition, $\underline{D}_r(\varepsilon)$ is a cover of $\underline{G}_{\alpha}$ and respectively, $\overline{D}_r(\varepsilon)$ is a cover of $\overline{G}_{\alpha}$. Now let $\underline{C}_r(\varepsilon)$ and $\overline{C}_r(\varepsilon)$ be a disjoint set of cylinders such that
    $$\bigcup_{[\iiv]\in\underline{D}_r(\varepsilon)}[\iiv] = \bigcup_{[\iiv]\in\underline{C}_r(\varepsilon)}[\iiv]\text{ and }\bigcup_{[\iiv]\in\overline{D}_r(\varepsilon)}[\iiv] = \bigcup_{[\iiv]\in\overline{C}_r(\varepsilon)}[\iiv].$$
    Then by \eqref{eq:Gibbs} and the definition of $\underline{C}_r(\varepsilon)$, for any $t\geq0$
    \begin{align*}
    \mathcal{H}^{\alpha t-P(t)+(1+t)\varepsilon}_{r}(\underline{G}_{\alpha}) &\leq \sum_{[\iiv]\in\underline{C}_r(\varepsilon)}\lambda_{\iiv}^{(\alpha t-P(t)+(1+t)\varepsilon)}\\
    &\leq \lambda_{\min}^{-1}r^{\varepsilon} \sum_{[\iiv]\in\underline{C}_r(\varepsilon)}\|A_{\iiv}\|^t\lambda_{\iiv}^{-P(t)}\\
    &\leq C\lambda_{\min}^{-1}r^{\varepsilon}\sum_{[\iiv]\in\underline{C}_r(\varepsilon)}\mu_t([\iiv])\leq C\lambda_{\min}^{-1}r^{\varepsilon}.
    \end{align*}
    Hence, $\mathcal{H}^{\alpha t-P(t)+(1+t)\varepsilon}(\underline{G}_{\alpha})=0$ for any $t>0$ and any $\varepsilon>0$, so \eqref{eq:ub1} follows. The proof of \eqref{eq:ub2} is similar by using the cover $\overline{C}_r(\varepsilon)$ of $\overline{G}_{\alpha}$.
\end{proof}

We note that by the concavity of $P$
$$
\inf_{t\in\R}\left\{t\alpha-P(t)\right\}=\inf_{t\leq 0}\left\{t\alpha-P(t)\right\},
$$
for every  and $\alpha\in[P'(0),\alpha_{\max}]$,
$$
\inf_{t\in\R}\left\{t\alpha-P(t)\right\}=\inf_{t\geq 0}\left\{t\alpha-P(t)\right\},
$$
for every $\alpha\in[\alpha_{\min},P'(0)]$.

\section{Pointwise H\"older exponent for non-degenerate curves}\label{sec:nondegZippers}

First, let us define the natural projections $\pi$ and $\Pi$ from
the symbolic space $\Sigma$ to the unit interval $[0,1]$ and the
curve $\Gamma$. We recall that we assumed that all the signatures of
the affine zipper Definition~\ref{def:zipper} is $0$, and all the
matrices are invertible. Therefore,
\begin{eqnarray}
\pi(\ii) & = & \lim_{n\to\infty}g_{i_1}\circ\cdots\circ g_{i_n}(0)=\sum_{n=1}^{\infty}\lambda_{\ii|_{n-1}}\gamma_{i_n}\\
\Pi(\ii) & = & \lim_{n\to\infty}f_{i_1}\circ\cdots\circ f_{i_n}(\underline{0})=\sum_{n=1}^{\infty}A_{\ii|_{n-1}}t_{i_n}.
\end{eqnarray}
Observe that by the definition of the linear parametrization $v$ of $\Gamma$, $v(\pi(\ii))=\Pi(\ii)$.

In the analysis of the pointwise H\"older exponent $\alpha$, defined in \eqref{eq:alpha1}, the points play important role which are far away symbolically but close on the self-affine curve. To be able to handle such points we introduce the following notation
\begin{equation*}\label{eq:defiveej}
\ii\vee\jj =
\left\{
\begin{array}{ll}
\min \{\sigma^{\ii\wedge\jj+1}\ii\wedge\1,\sigma^{\ii\wedge\jj+1}\jj\wedge\0\}, & \hbox{if $i_{\ii\wedge\jj+1}+1=j_{\ii\wedge\jj+1}$,} \\
\min \{\sigma^{\ii\wedge\jj+1}\ii\wedge\0,\sigma^{\ii\wedge\jj+1}\jj\wedge\1\}, & \hbox{if $j_{\ii\wedge\jj+1}+1=i_{\ii\wedge\jj+1}$,} \\
0, & \hbox{otherwise,}
\end{array}
\right.
\end{equation*}
where $\0$ denotes the $(0,0,\dots)$ and $\1$ denotes the
$(N-1,N-1,\dots)$ sequence. It is easy to see that there exists a
constant $K>0$ such that
    \begin{equation}\label{eq:bound}
    K^{-1}(\lambda_{\ii|_{\ii\wedge\jj+\ii\vee\jj}}+\lambda_{\jj|_{\ii\wedge\jj+\ii\vee\jj}})\leq|\pi(\ii)-\pi(\jj)|\leq K(\lambda_{\ii|_{\ii\wedge\jj+\ii\vee\jj}}+\lambda_{\jj|_{\ii\wedge\jj+\ii\vee\jj}}).
    \end{equation}
Hence, the distance on $[0,1]$ is not comparable with the distance on the symbolic space. More precisely, let $T$ be the set of points on the symbolic space, which has tail $0$ or $N-1$, i.e. $\ii\in T$ if and only if there exists a $k\geq0$ such that $\sigma^k\ii=\0$ or $\sigma^k\ii=\1$. So if $\pi(\sigma^k\ii)$ is too close to the set $\pi(T)$ infinitely often then we lose the symbolic control over the distance $|\pi(\ii)-\pi(\ii_n)|$, where $\ii_n$ is such that $\pi(\ii_n)\to\pi(\ii)$ as $n\to\infty$.

On the other hand, the symbolic control of the set $\|\Pi(\ii)-\Pi(\ii_n)\|$ is also far non-trivial. In general, $\|\Pi(\ii)-\Pi(\jj)\|=\|A_{\ii|_{\ii\wedge\jj}}(\Pi(\sigma^{\ii\wedge\jj}\ii)-\Pi(\sigma^{\ii\wedge\jj}\jj))\|$ is not comparable to $\|A_{\ii|_{\ii\wedge\jj}}\|\cdot\|\Pi(\sigma^{\ii\wedge\jj}\ii)-\Pi(\sigma^{\ii\wedge\jj}\jj)\|$, unless $\langle\Pi(\sigma^{\ii\wedge\jj}\ii)-\Pi(\sigma^{\ii\wedge\jj}\jj)\rangle\in M$, where $M$ is the multicone satisfying the Definition~\ref{def:domsplit}. Thus, in order to handle
$$
\liminf_{n\to\infty}\frac{\log\|\Pi(\ii)-\Pi(\ii_n)\|}{\log|\pi(\ii)-\pi(\ii_n)|}
$$
we need that $\ii$ is sufficiently far from the tail set $T$ and
also that the points $\Pi(\ii_n)$ on $\Gamma$ can be chosen such
that
$\langle\Pi(\sigma^{\ii\wedge\jj}\ii)-\Pi(\sigma^{\ii\wedge\ii_n}\ii_n)\rangle\in
M$. So we introduce a kind of exceptional set $B$, where both of
these requirements fail. We define $B\subseteq\Sigma$ such that
\begin{multline}\label{eq:setBad}
B=\left\{\ii\in\Sigma:\forall'\ n\geq1\ \forall'\ l\geq1\ \forall'\ m\geq1\ \exists'\ K\geq0\ \forall'\ k\geq K\right.\\
\left.\left(M(\Pi(\sigma^k\ii))\setminus B_{1/n}(\Pi(\sigma^k\ii))\right)\cap \Gamma\setminus(\Gamma_{\sigma^k\ii|_l}\cup\Gamma_{\sigma^k\ii|_{l-1}(i_{k+l}-1)(N-1)^m}\cup\Gamma_{\sigma^k\ii|_{l-1}(i_{k+l}+1)0^m})=\emptyset\right\},
\end{multline}
where $\Gamma_{\iiv}=f_{\iiv}(\Gamma)$ for any finite length word
$\iiv\in\Sigma^*$ and $M(\Pi(\ii))$ is the cone centered at
$\Pi(\ii)$. We note that if $i_{l}=0$ (or $i_{l}=N-1$) then we
define $\Gamma_{\sigma^k\ii|_{l-1}(i_{l}-1)(N-1)^m}=\emptyset$ (or
$\Gamma_{\sigma^k\ii|_{l-1}(i_{l}+1)0^m}=\emptyset$ respectively).

In particular, $B$ contains those points $\ii$, for which locally the curve $\Gamma$ will leave the cone $M$ very rapidly. In other words, let
\begin{multline*}
B_{n,l,m}=
\left\{\ii\in\Sigma:\right.\\ \left.\left(M(\Pi(\ii))\setminus B_{1/n}(\Pi(\ii))\right)\cap \Gamma\setminus(\Gamma_{\ii|_l}\cup\Gamma_{\ii|_{l-1}(i_{l}-1)(N-1)^m}\cup\Gamma_{\ii|_{l-1}(i_{l}+1)0^m})=\emptyset\right\}.
\end{multline*}
and
$$
B_{n,m,l,K}=\bigcap_{k=K}^{\infty}\sigma^{-k}B_{n,l,m}\text{ and }B=\bigcap_{n=1}^{\infty}\bigcap_{l=0}^{\infty}\bigcap_{m=0}^{\infty}\bigcup_{K=0}^{\infty}B_{n,l,m,K}.
$$

For a visualisation of the local neighbourhood of a point in
$B_{n,l,m}$, see Figure~\ref{fig:bnm}. In particular, we are able to
handle the pointwise H\"older exponents at $\pi(\ii)$ outside to the
set $B$ and we show that $B$ is small in some sense.

\begin{figure}
  \centering
  \includegraphics[width=80mm]{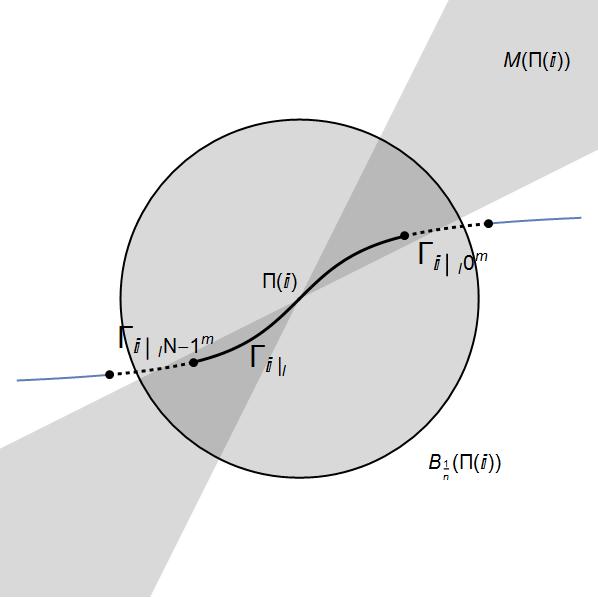}
  \caption{Local neighbourhood of points in $B_{n,l,m}$.}\label{fig:bnm}
\end{figure}

\begin{lemma}\label{lem:tobad}
    Let us assume that $\mathcal{S}$ is non-degenerate. Then there exist $n\geq1$, $l\geq1$, $m\geq1$ and $\jjv$ finite length word with $|\jjv|=l$, such that $$B_{n,l,m}\cap [\jjv]=\emptyset.$$
\end{lemma}

\begin{proof}
    Our first claim is that there exists a finite sequence $\iiv$ such that $\langle A_{\iiv}(z_N-z_0)\rangle\in M$. Suppose that this is not the case. That is, for every finite length word $\langle A_{\iiv}(z_N-z_0)\rangle\in M^c$. Equivalently, for every finite length word $\iiv$, $\langle z_N-z_0\rangle\in A_{\iiv}^{-1}(M^c)$. Thus, $\langle z_N-z_0\rangle\in \bigcap_{k=0}^{\infty}\bigcap_{|\iiv|=k} A_{\iiv}^{-1}(M^c)$, which contradicts to our non-degeneracy assumption.

Let us fix an $\iiv$ such that $\langle A_{\iiv}(z_N-z_0)\rangle\in
M$. Then $f_{\iiv}(z_N)\in M(f_{\iiv}(z_0))$. By continuity, one can
choose $k\geq1$ large enough such that for every $\ii\in[\iiv0^k]$,
$$
f_{\iiv}(z_N)\in M(\Pi(\ii))$$
and
$$
\|f_{\iiv}(z_0)-\Pi(\ii)\|=\|A_{\iiv0^k}(z_0-\Pi(\sigma^{|\iiv|+k}\ii))\|\leq\|A_{\iiv}\|\|A_0^k\|\mathrm{diam}(\Gamma)\leq\frac{1}{2}\|A_{\iiv}(z_N-z_0)\|,.
$$
where we used the fact that $f_0(z_0)=z_0$. Then
$$
\|\Pi(\ii)-f_{\iiv}(z_N)\|\geq\|A_{\iiv}(z_N-z_0)\|-\|f_{\iiv}(z_0)-\Pi(\ii)\|>\frac{1}{2}\|A_{\iiv}(z_N-z_0)\|.
$$
 We get that for every  $\ii\in[\iiv0^k]$
    $$
    f_{\iiv}(z_N)\in \left(M(\Pi(\ii))\setminus B_{\frac{1}{2}\|A_{\iiv}(z_N-z_0)\|}(\Pi(\ii))\right)\cap\Gamma\setminus(\Gamma_{\iiv0^k}\cup\Gamma_{\iiv0^{k-1}10}\cup\Gamma_{\iiv|_{|\iiv|-1}(\iiv_{|\iiv|}-1)N})\neq\emptyset.
    $$
    By fixing $\jjv:=\iiv0^k$, $l:=|\jjv|$, $m:=1$ and $n:=\left\lceil\frac{2}{\|A_{\iiv}(z_N-z_0)\|}\right\rceil$,  we see that $B_{n,l,m}\cap[\jjv]=\emptyset$.
\end{proof}

\begin{prop}\label{prop:Bdim}
    Let us assume that $\mathcal{S}$ is non-degenerate. Then $\dim_P\pi(B)<1$. Moreover, for any $\nu$ fully supported ergodic measure on $\Sigma$, $\nu(B)=0$.
\end{prop}

\begin{proof}
    By definition, $B_{n,l,m}\supseteq B_{n+1,l,m}$, $B_{n,l,m}\supseteq B_{n,l+1,m}$ and $B_{n,l,m}\supseteq B_{n,l,m+1}$. Moreover, $B_{n,l,m,K}=\sigma^{-K}B_{n,l,m,0}$. In particular, $\sigma^{-1}B_{n,l,m,0}=B_{n,l,m,1}\supseteq B_{n,l,m,0}$. Thus, for every $n\geq1$
    \begin{equation}\label{etech1}
    B_{n,l,m,0}\subseteq\bigcup_{|\iiv|=q}\varrho_{\iiv}(B_{n,l,m,0}),
    \end{equation}
   where $\varrho_{\iiv}(\ii)=\iiv\ii$. Let $n_0\geq1$, $l_0\geq1$, $m_0\geq1$ be natural numbers and $\jjv$ be the finite length word with $|\jjv|=l_0$  as in Lemma~\ref{lem:tobad}, then
   $$
   B_{n_0,l_0,m_0,0}\cap[\jjv]=\bigcap_{k=0}^{\infty}\left(\sigma^{-k}B_{n_0,m_0,l_0}\cap[\jjv]\right)\subseteq B_{n_0,m_0,l_0}\cap[\jjv]=\emptyset.
   $$
   Thus,
    \begin{equation}\label{etech2}
    B_{n_0,l_0,m_0,0}\subseteq\bigcup_{\stackrel{|\iiv|=l_0}{ \iiv\neq\jjv}}\varrho_{\iiv}(B_{n_0,l_0,m_0,0}).
    \end{equation}
    Hence, $\sigma^{p}\ii\notin[\jjv]$ for every $\ii\in B_{n_0,l_0,m_0,0}$ and for every $p\geq1$. Indeed, if there exist $\ii\in B_{n_0,l_0,m_0,0}$ and $p\geq1$ such that $\sigma^p\ii\in[\jjv]$ then there exist a finite length word $\iiv$ with $|\iiv|=p$ such that $B\cap[\iiv\jjv]\neq\emptyset$. But by equations \eqref{etech1} and \eqref{etech2},
    $$
    B_{n_0,l_0,m_0,0}\subseteq\bigcup_{|\iiv_1|=p}\varrho_{\iiv_1}(B_{n_0,l_0,m_0,0})\subseteq\bigcup_{|\iiv_1|=p}\bigcup_{\stackrel{|\iiv_2|=l_0}{ \iiv_2\neq\jjv}}\varrho_{\iiv_1}(\varrho_{\iiv_2}(B_{n_0,l_0,m_0,0}))\subseteq\bigcup_{|\iiv_1|=p}\bigcup_{\stackrel{|\iiv_2|=l_0}{ \iiv_2\neq\jjv}}[\iiv_1\iiv_2]
    $$
    which is a contradiction. But for any fully supported ergodic measure $\nu$, $\nu([\jjv])>0$ and therefore $\nu(B_{n_0,l_0,m_0,0})=0$. The second statement of the lemma follows by $$\nu(B)\leq\inf_{n,l,m}\nu(\bigcup_{K=0}^{\infty}B_{n,l,m,K})\leq\sum_{K=0}^{\infty}\nu(B_{n_0,l_0,m_0,K})=\sum_{K=0}^{\infty}\nu(B_{n_0,l_0,m_0,0})=0.$$

    To prove the first assertion of the proposition, observe that by equation~\eqref{etech2}
    $$
    \pi(B_{n_0,l_0,m_0,0})\subseteq\bigcup_{\stackrel{|\iiv|=l_0}{ \iiv\neq\jjv}}g_{\iiv}(\pi(B_{n_0,l_0,m_0,0})).
    $$
    Therefore, $\pi(B_{n_0,l_0,m_0,0})$ is contained in the attractor $\Lambda$ of the IFS $\left\{g_{\iiv}\right\}_{\stackrel{|\iiv|=l_0}{ \iiv\neq\jjv}}$, for which $\dim_B\Lambda<1$.
Hence,
$$
    \dim_P\pi(B)\leq\inf_{n,l,m}\overline{\dim}_B\pi(B_{n,l,m,0})\leq\overline{\dim}_B\pi(B_{n_0,l_0,m_0,0})\leq\dim_B\Lambda<1.
    $$

\end{proof}

\begin{lemma}\label{lem:ub}
Let us assume that $\mathcal{S}$ is non-degenerate. Then for every $\ii\in\Sigma\setminus B$
$$\alpha(\pi(\ii))\leq\limsup_{n\to+\infty}\frac{\log\|A_{\ii|_n}\|}{\log\lambda_{\ii|_n}}.$$
\end{lemma}

\begin{proof}
    Let $\ii\in\Sigma\setminus B$. Then there exist $n\geq1$, $l\geq1$, $m\geq1$ and a sequence $\left\{k_p\right\}_{p=1}^{\infty}$ such that $k_p\to\infty$ as $p\to\infty$ and
\begin{multline}\label{eq:antiB}
    \left(M(\Pi(\sigma^{k_p}\ii))\setminus B_{1/n}(\Pi(\sigma^{k_p}\ii))\right)\cap\\ \Gamma\setminus(\Gamma_{\sigma^{k_p}\ii|_l}\cup\Gamma_{\sigma^{k_p}\ii|_{l-1}(i_{{k_p}+l}-1)(N-1)^m}\cup\Gamma_{\sigma^{k_p}\ii|_{l-1}(i_{k_p+l}+1)0^m})\neq\emptyset
    \end{multline}
    Hence, there exists a sequence $\jj_p$ such that $k_p\leq\ii\wedge\jj_p\leq k_p+l$, $\ii\vee\jj_p\leq m$,
\begin{equation}\label{eq:antiC}
  \Pi(\sigma^{k_p}\jj_p)\in M(\Pi(\sigma^{k_p}\ii))\text{ and }\|\Pi(\sigma^{k_p}\jj_p)-\Pi(\sigma^{k_p}\ii)\|>\frac{1}{n}.
\end{equation}
Thus,
    \begin{multline*}
    \alpha(\pi(\ii))=\liminf_{\pi(\jj)\to\pi(\ii)}\frac{\log\|\Pi(\ii)-\Pi(\jj)\|}{\log|\pi(\ii)-\pi(\jj)|}\leq \liminf_{p\to+\infty}\frac{\log\|\Pi(\ii)-\Pi(\jj_p)\|}{\log|\pi(\ii)-\pi(\jj_p)|}=\\
    \liminf_{p\to+\infty}\frac{\log\|A_{\ii|_{k_p}}(\Pi(\sigma^{k_p}\ii)- \Pi(\sigma^{k_p}\jj_p))\|}{\log|\lambda_{\ii|_{\ii\wedge\jj_p+\ii\vee\jj_p}}(\pi(\sigma^{\ii\wedge\jj_p+\ii\vee\jj_p}\ii)- \pi(\sigma^{\ii\wedge\jj_p+\ii\vee\jj_p}\jj_p))|},
    \end{multline*}
     and by \eqref{eq:compare}, \eqref{eq:antiC},
     \begin{multline*}
     \liminf_{p\to+\infty}\frac{\log\|A_{\ii|_{k_p}}(\Pi(\sigma^{k_p}\ii)-\Pi(\sigma^{k_p}\jj_p))\|}{\log|\lambda_{\ii|_{\ii\wedge\jj_p+\ii\vee\jj_p}}(\pi(\sigma^{\ii\wedge\jj_p+\ii\vee\jj_p}\ii)-\pi(\sigma^{\ii\wedge\jj_p+\ii\vee\jj_p}\jj_p))|}\leq\\
     \liminf_{p\to+\infty}\frac{\log(C^{-1}/n)+\log\|A_{\ii|_{k_p}}\|}{\log\lambda_{\ii|_{k_p}}+\log d'}\leq\limsup_{p\to+\infty}\frac{\log\|A_{\ii|_p}\|}{\log\lambda_{\ii|_p}},
     \end{multline*}
     where $d'=(\max_i\lambda_i)^{m+l}$.
\end{proof}

\begin{lemma}\label{lem:erg}
    Let us assume that $\mathcal{S}$ is non-degenerate. Then for every ergodic, $\sigma$-invariant, fully supported measure $\mu$ on $\Sigma$ such that $\sum_{k=0}^{\infty}(\mu[0^k]+\mu([N^k])$ is finite, then
    $$\alpha(\pi(\ii))=\lim_{n\to+\infty}\frac{\log\|A_{\ii|_n}\|}{\log\lambda_{\ii|_n}}\text{ for $\mu$-a.e. $\ii\in\Sigma$.}$$
\end{lemma}

\begin{proof}
    By Proposition~\ref{prop:Bdim}, we have that $\mu(B)=0$. Thus, by Lemma~\ref{lem:ub}, for $\mu$-a.e. $\ii$
    $$
    \alpha(\pi(\ii))\leq\lim_{n\to+\infty}\frac{\log\|A_{\ii|_n}\|}{\log\lambda_{\ii|_n}}.
    $$

    On the other hand, for every $\ii\in\Sigma$,
    $$
    \alpha(\pi(\ii))=\liminf_{\pi(\jj)\to\pi(\ii)}\frac{\log\|\Pi(\ii)-\Pi(\jj)\|}{\log|\pi(\ii)-\pi(\jj)|}\geq\liminf_{\jj\to\ii}\frac{\log\|A_{\ii|_{\ii\wedge\jj}}\|}{\log\lambda_{\ii|_{\ii\wedge\jj+\ii\vee\jj}}+\log \min_{i}\lambda_i}.
    $$

    Hence, to verify the statement of the lemma, it is enough to show that $$\lim_{\jj\to\ii}\frac{\log\lambda_{\ii|_{\ii\wedge\jj}}}{\log\lambda_{\ii|_{\ii\wedge\jj+\ii\vee\jj}}}=1\text{ for $\mu$-a.e. $\ii$.}
    $$
    It is easy to see that from $\lim_{\jj\to\ii}\frac{\ii\vee\jj}{\ii\wedge\jj}=0$ follows the previous equation. Let
    $$
    R_{n}=\left\{\ii:\exists\jj_k\text{ s. t. }\jj_k\to\ii\text{ as }k\to\infty\text{ and }\lim_{k\to\infty}\frac{\ii\vee\jj_k}{\ii\wedge\jj_k}>\frac{1}{n}\right\}
    $$
    In other words,
    $$
    R_n=\bigcap_{K=0}^{\infty}\bigcup_{k=K}^{\infty}\bigcup_{i_1,\dots,i_k=0}^N[i_1,\dots,i_k,\overbrace{0,\dots,0}^{\lfloor k/n\rfloor}]\cup[i_1,\dots,i_k,\overbrace{N,\dots,N}^{\lfloor k/n\rfloor}]
    $$
    Therefore, for any $\mu$ ergodic $\sigma$-invariant measure and for every $K\geq0$
    $$
    \mu(R_n)\leq\sum_{k=K}^{\infty}(\mu([0^{\lfloor k/n\rfloor}])+\mu([N^{\lfloor k/n\rfloor}])).
    $$
    Since by assumption the sum on the right hand side is summable, we get $\mu(R_n)=0$ for every $n\geq1$.
\end{proof}

Let us recall that we call the function $v:[0,1]\mapsto\Gamma$
\textit{symmetric} if
\begin{equation}
\lambda_0=\lambda_{N-1}\text{ and }\lim_{k\to\infty}\frac{\|A_0^{k}\|}{\|A_{N-1}^k\|}=1.
\end{equation}

\begin{lemma}\label{lem:lb}
Let us assume that $\mathcal{S}$ is non-degenerate and symmetric. Then for every $\ii\in\Sigma$
$$\alpha(\pi(\ii))\geq\liminf_{n\to+\infty}\frac{\log\|A_{\ii|_n}\|}{\log\lambda_{\ii|_n}}.$$
\end{lemma}

\begin{proof}
Let us observe that by the zipper property
$f_i(\Pi(\0))=f_{i-1}(\1)$ for every $1\leq i\leq N-1$. Moreover,
for any $\ii,\jj$ with $i_{\ii\wedge\jj+1}=j_{\ii\wedge\jj+1}+1$,
$$\ii\vee\jj=\min\{\sigma^{\ii\wedge\jj+1}\ii\wedge\1,\sigma^{\ii\wedge\jj+1}\jj\wedge\0\}.$$
Thus, if $i_{\ii\wedge\jj+1}=j_{\ii\wedge\jj+1}+1$
\begin{align}
\|\Pi(\ii)-\Pi(\jj)\|&=\|\Pi(\ii)-\Pi(\ii|_{\ii\wedge\jj}i_{\ii\wedge\jj+1}\0)+\Pi(\jj|_{\ii\wedge\jj}j_{\ii\wedge\jj+1}\1)-\Pi(\jj)\|\nonumber\\
&=\|A_{\ii|_{\ii\wedge\jj+\ii\vee\jj}}(\Pi(\sigma^{\ii\wedge\jj+\ii\vee\jj}\ii)-\Pi(\0))+A_{\ii|_{\ii\wedge\jj+\ii\vee\jj}}(\Pi(\1)-\Pi(\sigma^{\ii\wedge\jj+\ii\vee\jj}\jj))\|\nonumber\\
&\leq(\|A_{\ii|_{\ii\wedge\jj+\ii\vee\jj}}\|+\|A_{\jj|_{\ii\wedge\jj+\ii\vee\jj}}\|)\mathrm{diam}(\Gamma).\label{eq:this}
\end{align}
The case $i_{\ii\wedge\jj+1}=j_{\ii\wedge\jj+1}-1$ is similar, and
if $|i_{\ii\wedge\jj+1}-j_{\ii\wedge\jj+1}|\neq1$ then
$\ii\vee\jj=0$, so \eqref{eq:this} holds trivially. Moreover by
\eqref{eq:bound}, there exist constants $K_1,K_2>0$ such that for
every $\ii,\jj\in\Sigma$
    \begin{equation}\label{eq:localH}
    \frac{\log\|\Pi(\ii)-\Pi(\jj)\|}{\log|\pi(\ii)-\pi(\jj)|}\geq\frac{-\log K_1+\log(\|A_{\ii|_{\ii\wedge\jj+\ii\vee\jj}}\|+\|A_{\jj|_{\ii\wedge\jj+\ii\vee\jj}}\|)}{\log(\lambda_{\ii|_{\ii\wedge\jj+\ii\vee\jj}}+\lambda_{\jj|_{\ii\wedge\jj+\ii\vee\jj}})+\log K_2}.
    \end{equation}
Therefore,
\begin{multline*}
\alpha(\pi(\ii))=\liminf_{\pi(\jj)\to\pi(\ii)}\frac{\log\|\Pi(\ii)-\Pi(\jj)\|}{\log|\pi(\ii)-\pi(\jj)|}\geq\\
\liminf_{\jj\to\ii}\frac{-\log K_1+\log(\|A_{\ii|_{\ii\wedge\jj+\ii\vee\jj}}\|+\|A_{\jj|_{\ii\wedge\jj+\ii\vee\jj}}\|)}{\log(\lambda_{\ii|_{\ii\wedge\jj+\ii\vee\jj}}+\lambda_{\jj|_{\ii\wedge\jj+\ii\vee\jj}})+\log K_2}=\\
\liminf_{\jj\to\ii}\frac{\log\|A_{\ii|_{\ii\wedge\jj+\ii\vee\jj}}\|\left(-\frac{\log K_1}{\log\|A_{\ii|_{\ii\wedge\jj+\ii\vee\jj}}\|}+1+\frac{\log\left(1+\frac{\|A_{\jj|_{\ii\wedge\jj+\ii\vee\jj}}\|}{\|A_{\ii|_{\ii\wedge\jj+\ii\vee\jj}}\|}\right)}{\log\|A_{\ii|_{\ii\wedge\jj+\ii\vee\jj}}\|}\right)}{\log\lambda_{\ii|_{\ii\wedge\jj+\ii\vee\jj}}\left(1+\frac{\log 2K_2}{\log\lambda_{\ii|_{\ii\wedge\jj+\ii\vee\jj}}}\right)}.
\end{multline*}
So, to verify the statement of the lemma, it is enough to show that
there exists a constant $C>0$ such that for every,
$\ii,\jj\in\Sigma$
$$
C^{-1}\leq\frac{\|A_{\jj|_{\ii\wedge\jj+\ii\vee\jj}}\|}{\|A_{\ii|_{\ii\wedge\jj+\ii\vee\jj}}\|}\leq C.
$$
By Theorem~\ref{thm:BochiGour}\eqref{item:bochigour5} and \eqref{eq:multip}, there exist $C'>0$ such that
\begin{align*}
\|A_{\ii|_{\ii\wedge\jj+\ii\vee\jj}}\|&\geq \|A_{\ii|_{\ii\wedge\jj+\ii\vee\jj}}|E(\ii')\|=\|A_{\ii|_{\ii\wedge\jj}}|E(\sigma^{\ii\wedge\jj}\ii|_{\ii\vee\jj}\ii')\|\|A_{\sigma^{\ii\wedge\jj}\ii|_{\ii\vee\jj}}|E(\ii')\|\\
&\geq C'\|A_{\ii|_{\ii\wedge\jj}}\|\|A_{\sigma^{\ii\wedge\jj}\jj|_{\ii\vee\jj}}\|
\end{align*}
and
$$
\|A_{\jj|_{\ii\wedge\jj+\ii\vee\jj}}\|\leq\|A_{\jj|_{\ii\wedge\jj}}\|\|A_{\sigma^{\ii\wedge\jj}\jj|_{\ii\vee\jj}}\|
$$
clearly. The other bounds are similar. But if
$i_{\ii\wedge\jj+1}=j_{\ii\wedge\jj+1}+1$ then
$\|A_{\sigma^{\ii\wedge\jj}\jj|_{\ii\vee\jj}}\|=\|A_0^{\ii\vee\jj}\|$
and
$\|A_{\sigma^{\ii\wedge\jj}\ii|_{\ii\vee\jj}}\|=\|A_{N-1}^{\ii\vee\jj}\|$.
Thus, by \eqref{esymm},
$$
\alpha(\pi(\ii))\geq \liminf_{\jj\to\ii}\frac{\log\|A_{\ii|_{\ii\wedge\jj+\ii\vee\jj}}\|}{\log\lambda_{\ii|_{\ii\wedge\jj+\ii\vee\jj}}}\geq\liminf_{n\to+\infty}\frac{\log\|A_{\ii|_n}\|}{\log\lambda_{\ii|_n}}.
$$

\end{proof}

\begin{proof}[Proof of Theorem~\ref{thm:nondeg}]
    First, we show that for $\mathcal{L}$-a.e. $x$, the local H\"older exponent is a constant. Since $\mu_0=\left\{\lambda_1,\dots,\lambda_N\right\}^{\N}$, it is easy to see that $\pi_*\mu_0=\mathcal{L}|_{[0,1]}$. Thus, it is enough to show that for $\mu_0$-a.e. $\ii\in\Sigma$, $\alpha(\pi(\ii))$ is a constant.

But by Proposition~\ref{prop:pressure}, there exists $\widehat{\alpha}$ such that for $\mu_0$-a.e. $\ii$
    $$
    \widehat{\alpha}=\lim_{n\to+\infty}\frac{\log\|A_{\ii|_n}\|}{\log\lambda_{\ii|_n}}.
    $$
    By definition of Bernoulli measure, $\sum_{k=0}^{\infty}\mu_0([0^k])+\mu_0([N^k])=\frac{1}{1-\lambda_1}+\frac{1}{1-\lambda_N}$. Thus, by Lemma~\ref{lem:erg}, $\alpha(\pi(\ii))=\widehat{\alpha}$ for $\mu_0$-a.e. $\ii$, and by Lemma~\ref{lem:boundsonder}, we have $\widehat{\alpha}\geq1/d_0$.

    We show now the lower bound for \eqref{eq:nondegmulti}. By Lemma~\ref{lem:pressure}, the map $t\mapsto P'(t)$ is continuous and monotone increasing on $\R$. Hence, for every $\beta\in(\alpha_{\min},\alpha_{\max})$ there exists a $t_0\in\R$ such that $P'(t_0)=\beta$. By Proposition~\ref{prop:pressure}, there exists a $\mu_{t_0}$ Gibbs measure on $\Sigma$ such that for $$
    \lim_{n\to+\infty}\frac{\log\|A_{\ii|_n}\|}{\log\lambda_{\ii|_n}}=\beta\text{ for $\mu_{t_0}$-a.e. }\ii\in\Sigma.
    $$
    It is easy to see that for any $\ii$ and $n\geq1$, $\mu_{t_0}([\ii|_n])>0$. Thus, by Lemma~\ref{lem:erg},
    $$
    \alpha(\pi(\ii))=\beta\text{ for $\mu_{t_0}$-a.e. }\ii\in\Sigma.
    $$
    Therefore, by Proposition~\ref{prop:pressure}
    \begin{multline*}
      \dim_H\left\{x\in[0,1]:\alpha(x)=\beta\right\}\geq\dim_H\mu_{t_0}\circ\pi^{-1}=t_0P'(t_0)-P(t_0)=\\
      t_0\beta-P(t_0)\geq\inf_{t\in\R}\left\{t\beta-P(t)\right\}.
      \end{multline*}

    On the other hand, by Lemma~\ref{lem:ub}
    \begin{align*}
    \dim_H\left\{x\in[0,1]:\right.&\left.\alpha(x)=\beta\right\}\\
    &\leq\max\left\{\dim_H\pi(B),\dim_H\left\{\ii\in\Sigma\setminus B:\alpha(\pi(\ii))=\beta\right\}\right\}\\
    &\leq\max\left\{\dim_H\pi(B),\dim_H\left\{\ii\in\Sigma:\limsup_{n\to+\infty}\frac{\log\|A_{\ii|_n}\|}{\log\lambda_{\ii|_n}}\geq\beta\right\}\right\}\\
    &\leq\max\left\{\dim_H\pi(B),\inf_{t\leq 0}\left\{t\beta-P(t)\right\}\right\},
    \end{align*}
    where in the last inequality we used Lemma~\ref{lem:upperbound}.

By Proposition~\ref{prop:pressure}, the function $t\mapsto tP'(t)-P(t)$ is continuous and $P(0)=-1$. By Proposition~\ref{prop:Bdim}, $\dim_H\pi(B)<1$, thus, there exists an open neighbourhood of $t=0$ such that for every $t\in(-\rho,\rho)$, $tP'(t)-P(t)>\dim_P\pi(B)$. In other words, there exists a $\varepsilon>0$ such that $P'(t)\in(\widehat{\alpha}-\varepsilon,\widehat{\alpha}+\varepsilon)$ for every $t\in(-\rho,\rho)$. Hence, for every $\beta\in[\widehat{\alpha},\widehat{\alpha}+\varepsilon]$ there exists a $t_0\leq0$ such that $P'(t_0)=\beta$ and $\inf_{t\leq 0}\left\{t\beta-P(t)\right\}=t_0P'(t_0)-P(t_0)>\dim_H\pi(B)$ which completes the proof of \eqref{eq:nondegmulti}.

Finally, if \eqref{esymm} holds then by Lemma~\ref{lem:lb} and Lemma~\ref{lem:upperbound}
\begin{multline*}
    \dim_H\left\{x\in[0,1]:\alpha(x)=\beta\right\}\leq\\ \dim_H\left\{\ii\in\Sigma:\liminf_{n\to+\infty}\frac{\log\|A_{\ii|_n}\|}{\log\lambda_{\ii|_n}}\leq\beta\right\}\leq\inf_{t\geq 0}\left\{t\beta-P(t)\right\},
    \end{multline*}
which completes the proof.
\end{proof}

\section{Zippers with Assumption A}\label{sec:AssAZippers}

Now, we turn to the case when our affine zipper satisfies the
Assumption A. We will show that in fact in this case the exceptional
set $B$, introduced in \eqref{eq:setBad} is empty. That is, there
are no points, in which local neighbourhood, the curve leaves the
cone rapidly. First, let us introduce a natural ordering on
$\Sigma^*$. For any $\iiv,\jjv\in\Sigma^*$ with $\iiv\wedge\jjv=m$,
let
$$\iiv< \jjv \Leftrightarrow i_{m+1}<j_{m+1}.$$
Moreover, let $Z_1:=\left\{z_0,\dots, z_N\right\}$ the endpoints of
the curves $f_i(\Gamma)$ and let $Z_n:=\left\{f_{\iiv}(z_k),
|\iiv|=n, k=0,\ldots,N-1\right\}$.

For simplicity, let us denote $f_{\iiv}(z_0)$ by $z_{\iiv}$. Observe that by the Zipper property $f_{\iiv}(z_N)=z_{\iiv|_{|\iiv|-1}(i_{|\iiv|}-1)}$.

\begin{prop}\label{prop:AssAdimB}
Let us assume that $\mathcal{S}$ is non-degenerate and satisfies the Assumption A. Then $B=\emptyset$, where the set $B$ is defined in \eqref{eq:setBad}.
\end{prop}

\begin{proof}
It is enough to show that for every $\ii\in\Sigma$
\begin{equation}\label{eq:kell}
C(\Pi(\ii))\cap\Gamma=\Gamma,
\end{equation}
which is equivalent to show that for every $\ii,\jj\in\Sigma$, $\langle\Pi(\ii)-\Pi(\jj)\rangle\in C$.

Since $\langle z_0-z_N\rangle\in C$ and $C$ is invariant w.r.t all of the matrices then for every $\iiv\in\Sigma^*$, $\langle z_{\iiv}-z_{\iiv|_{|\iiv|-1}(i_{|\iiv|}-1)}\rangle=\langle f_{\iiv}(z_0)-f_{\iiv}(z_N)\rangle=\langle A_{\iiv}(z_0-z_N)\rangle\in C$.

Observe by convexity of $C$, for any three vectors $x,y,w\in\R^d$,
if $\langle x-y\rangle\in C$ and $\langle y-w\rangle\in C$ then
$\langle x-w\rangle\in C$. Thus, by Assumption A and the convexity
of the cone, for every $n\geq1$, and for every $\iiv<\jjv\in\Sigma$
with $|\iiv|=|\jjv|=n$, $\langle z_{\iiv}-z_{\jjv}\rangle\in C$.

Thus, for every $\ii\neq\jj\in\Sigma$ and for every $n\geq1$, $\langle f_{\ii|_n}(z_0)-f_{\jj|_n}(z_0)\rangle=\langle z_{\ii|_n}-z_{\jj|_n}\rangle\in C$. Since $C$ is closed, by taking $n$ tends to infinity, we get that $\langle \Pi(\ii)-\Pi(\jj)\rangle\in C$.
\end{proof}

\begin{lemma}\label{lem:ubassA}
Let us assume that $\mathcal{S}$ is non-degenerate and satisfies the Assumption A. Then for any $\mu$ fully supported, ergodic, $\sigma$-invariant measure on $\Sigma$
$$
\limsup_{\pi(\jj)\to\pi(\ii)}\frac{\log\|\Pi(\ii)-\Pi(\jj)\|}{\log|\pi(\ii)-\pi(\jj)|}\leq\lim_{n\to+\infty}\frac{\log\|A_{\ii|_n}\|}{\log\lambda_{\ii|_n}}\text{ for $\mu$-a.e. $\ii$.} $$
\end{lemma}

\begin{proof}
    Observe that
    \begin{multline*}
    \limsup_{\pi(\jj)\to\pi(\ii)}\frac{\log\|\Pi(\ii)-\Pi(\jj)\|}{\log|\pi(\ii)-\pi(\jj)|}= \\
    \limsup_{\pi(\jj)\to\pi(\ii)}\frac{\log\|A_{\ii|_{\ii\wedge\jj+\ii\vee\jj}}\left(\Pi(\sigma^{\ii\wedge\jj+\ii\vee\jj}\ii)- \Pi(\0)\right)+ A_{\jj|_{\ii\wedge\jj+\ii\vee\jj}}\left(\Pi(\1)-\Pi(\sigma^{\ii\wedge\jj+\ii\vee\jj}\jj)\right)\|}{|\lambda_{\ii|_{\ii\wedge\jj}}\left(\lambda_{\ii|_{\ii\vee\jj}}(\pi(\sigma^{\ii\wedge\jj+\ii\vee\jj}\ii)-0) +\lambda_{\jj|_{\ii\vee\jj}}(1-\pi(\sigma^{\ii\wedge\jj+\ii\vee\jj}\jj))\right)|}.
    \end{multline*}
By \eqref{eq:kell}, $\langle\Pi(\sigma^{\ii\wedge\jj+\ii\vee\jj}\ii)- \Pi(\0)\rangle, \langle\Pi(\1)-\Pi(\sigma^{\ii\wedge\jj+\ii\vee\jj}\jj)\rangle\in C$, therefore by \eqref{eq:compare}
$$
    \limsup_{\pi(\jj)\to\pi(\ii)}\frac{\log\|\Pi(\ii)-\Pi(\jj)\|}{\log|\pi(\ii)-\pi(\jj)|}\leq
\limsup_{\jj\to\ii}\frac{\log\|A_{\ii|_{\ii\wedge\jj+\ii\vee\jj}}\|+\log\left(1+\frac{\|A_{0}^{\ii\vee\jj}\|}{\|A_{N-1}^{\ii\vee\jj}\|}\right)}{\log\lambda_{\ii|_{\ii\wedge\jj}}}.
$$
It is easy to see that for any fully supported, ergodic,
$\sigma$-invariant measure $\mu$,
$\limsup_{\jj\to\ii}\frac{\ii\vee\jj}{\ii\wedge\jj+\ii\vee\jj}=0$
for $\mu$-a.e. $\ii$. Hence, by the previous inequality, the
statement follows similarly as in Lemma~\ref{lem:erg} .
\end{proof}

\begin{proof}[Proof of Theorem~\ref{thm:assA}.]
By Lemma~\ref{lem:erg} and Lemma~\ref{lem:ubassA}, for every $t\in\R$
$$
\alpha_r(\pi(\ii))=\lim_{n\to+\infty}\frac{\log\|A_{\ii|_n}\|}{\log\lambda_{\ii|_n}}\text{ for $\mu_{t}$-a.e. }\ii\in\Sigma.
$$
Thus, similarly to the proof of Theorem~\ref{thm:nondeg}
\begin{multline*}
      \dim_H\left\{x\in[0,1]:\alpha_r(x)=\beta\right\}\geq\dim_H\mu_{t_0}\circ\pi^{-1}=t_0P'(t_0)-P(t_0)=\\
      t_0\beta-P(t_0)\geq\inf_{t\in\R}\left\{t\beta-P(t)\right\},
      \end{multline*}
      where $t_0$ is defined such that $P'(t_0)=\beta$. On the other hand,
      $$
      \dim_H\left\{x\in[0,1]:\alpha_r(x)=\beta\right\}\leq\dim_H\left\{x\in[0,1]:\alpha(x)=\beta\right\}.
      $$
      By Proposition~\ref{prop:AssAdimB}, $B=\emptyset$, and similarly to the proof of Theorem~\ref{thm:nondeg},
      $$
      \dim_H\left\{x\in[0,1]:\alpha(x)=\beta\right\}\leq\inf_{t\leq 0}\left\{t\beta-P(t)\right\},
      $$
      for every $\beta\in[\hat{\alpha},\alpha_{\max}]$. If $\Gamma$ is symmetric then by Lemma~\ref{lem:lb} and Lemma~\ref{lem:upperbound}
      $$
      \dim_H\left\{x\in[0,1]:\alpha(x)=\beta\right\}\leq\inf_{t\geq 0}\left\{t\beta-P(t)\right\},
      $$
      which completes the proof.
\end{proof}

Now, we turn to the equivalence of the existence of pointwise
regular H\"older exponents and the Assumption~A. Before that, we
introduce another property and we show that in fact all of them are
equivalent. Denote $\mathrm{cv}(a,b)$ open line segment in $\R^d$
connecting two points $a,b$. Moreover, let us denote the orthogonal
projection to a subspace $\theta$ by $\proj_{\theta}$ and for a
subspace $\theta$ let $\theta^{\perp}$ be the orthogonal complement
of $\theta$. For a point $x$ and a subspace $\theta$, let
$\theta(x)=\{y\in\R^d:x-y\in\theta\}$.

\begin{definition}\label{def:wellord}
    We say that $Z_n$ is well ordered on $l\in G(d,d-1)$ if for any $\iiv_1<\iiv_2<\iiv_3$
    \begin{equation}\label{eq:wellordered}
    \mathrm{proj}_{l^\perp}(z_{\iiv_2}) \in \mathrm{cv}\big( \mathrm{proj}_{l^\perp}(z_{\iiv_1}), \mathrm{proj}_{l^\perp}(z_{\iiv_3}) \big).
    \end{equation}
We say that $Z_n$ is well ordered if there exists a $\delta>0$ such that $Z_n$ is well ordered for all $l\in B_{\delta}(F(\Sigma))$.
\end{definition}

Let us recall that $F:\Sigma\mapsto G(d,d-1)$ is the
H\"older-continuous function defined in Theorem~\ref{thm:BochiGour}.
So $B_{\delta}(F(\Sigma))$ is the $\delta>0$ neighbourhood of all
the possible subspaces on which the growth rate of the matrices is
at most the second singular value. For a visualisation of the well
ordered property, see Figure~\ref{fig:wellord}. Roughly speaking,
the well ordered property on $l\in G(d,d-1)$ means that the curve is
parallel to $l^{\perp}$. The next lemma indeed verifies that the
curve cannot turn back along $l^{\perp}$.

\begin{figure}
  \centering
  \includegraphics[width=80mm]{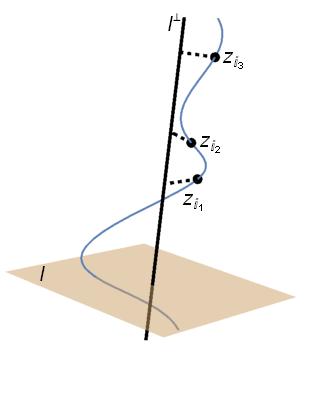}
  \caption{Well ordered property of $Z_n$ on $l\in G(d,d-1)$.}\label{fig:wellord}
\end{figure}

\begin{lemma}\label{lem:noturn}
    $Z_n$ is well ordered if and only if $\exists\delta>0,\forall \xv\in\R^d, \forall l\in B_{\delta}(F(\Sigma))$ either $\mathrm{cv}(z_{\iiv_1},z_{\iiv_2})\cap l(x)= \emptyset$ or $\mathrm{cv}(z_{\iiv_2},z_{\iiv_3})\cap l(x)=\emptyset$, for every $\iiv_1<\iiv_2<\iiv_3$.
\end{lemma}

\begin{proof}
Fix $\iiv_1<\iiv_2<\iiv_3\in\{0,\dots,N-1\}^n$. Suppose that $Z_n$ is well ordered but for all $\delta>0$ there exists $l\in B_{\delta}(F(\Sigma))$ and $\xv\in\R^d$ such that $\mathrm{cv}(z_{\iiv_1},z_{\iiv_2})\cap l(x)\neq \emptyset$ and $\mathrm{cv}(z_{\iiv_2},z_{\iiv_3})\cap l(x)\neq \emptyset$. Thus, $$\proj_{l^\perp}(x)\in\mathrm{cv}\big( \mathrm{proj}_{l^\perp}(z_{\iiv_2}), \mathrm{proj}_{l^\perp}(z_{\iiv_1}) \big)\cap\mathrm{cv}\big( \mathrm{proj}_{l^\perp}(z_{\iiv_2}), \mathrm{proj}_{l^\perp}(z_{\iiv_3}) \big)
$$
Since the right hand side is open, and non-empty line segment, therefore $$\mathrm{proj}_{l^\perp}(z_{\iiv_2})\notin\mathrm{cv}\big( \mathrm{proj}_{l^\perp}(z_{\iiv_1}), \mathrm{proj}_{l^\perp}(z_{\iiv_3}) \big),$$ which is a contradiction.

On the other hand, suppose that $Z_n$ satisfy the assumption of the lemma but not well ordered. Then for every $\delta>0$ there exists an $l\in B_{\delta}(F(\Sigma))$ such that $\mathrm{proj}_{l^\perp}(z_{\iiv_2}) \notin \mathrm{cv}\big( \mathrm{proj}_{l^\perp}(z_{\iiv_1}), \mathrm{proj}_{l^\perp}(z_{\iiv_3}) \big).$ Since $B_{\delta}(F(\Sigma))$ is open, there exists an $l'\in B_{\delta}(F(\Sigma))$ for which
$$
\mathrm{dist}(\mathrm{proj}_{l'^\perp}(z_{\iiv_2}),\mathrm{cv}\big( \mathrm{proj}_{l'^\perp}(z_{\iiv_1}), \mathrm{proj}_{l'^\perp}(z_{\iiv_3}) \big))>0.
$$
Thus, there exists $x\in\R^d$ that $\mathrm{cv}(z_{\iiv_1},z_{\iiv_2})\cap l'(x)\neq\emptyset$ and $\mathrm{cv}(z_{\iiv_3},z_{\iiv_2})\cap l'(x)\neq\emptyset$, which is again a contradiction.
\end{proof}

The next lemma gives us a method to check the well ordered property.

\begin{lemma}\label{lem:iff}
    $Z_0$ is well ordered if and only if for every $n\geq0$ $Z_n$ is well ordered.
\end{lemma}
\begin{proof}
The if part is trivial.

By definition, $Z_n=\bigcup_{k=0}^{N-1}f_k(Z_{n-1})$. By
Lemma~\ref{lem:noturn}, if $Z_{n-1}$ is well ordered then there
exists $\delta>0$ such that for every $l\in B_{\delta}(F(\Sigma))$
and for every $x\in\R^d$  either
$\mathrm{cv}(z_{\iiv_1},z_{\iiv_2})\cap l(x)= \emptyset$ or
$\mathrm{cv}(z_{\iiv_2},z_{\iiv_3})\cap l(x)=\emptyset$. Thus, in
particular for every $l\in B_{\delta}(F([k]))$. By
Theorem~\ref{thm:BochiGour}\eqref{item:BochiGour2}, there exists
$\delta'>0$ such that $A_kB_{\delta}(F([k]))\supseteq
B_{\delta'}(F(\Sigma))$. Thus, for every $l\in
B_{\delta'}(F(\Sigma))$ and for every $x\in\R^d$,
\begin{equation}\label{eq:fzn}
\text{either $\mathrm{cv}(f_k(z_{\iiv_1}),f_k(z_{\iiv_2}))\cap l(x)= \emptyset$ or $\mathrm{cv}(f_k(z_{\iiv_2}),f_k(z_{\iiv_3}))\cap l(x)=\emptyset$ }
\end{equation}
for every $z_{\iiv_1}, z_{\iiv_2}, z_{\iiv_3}\in Z_{n-1}$ with $\iiv_1<\iiv_2<\iiv_3$.

Let us suppose that $Z_n$ is not well ordered for some $n$. Hence, there exists a minimal $n$ such that $Z_{n-1}$ is well ordered but $Z_{n}$ is not. By Lemma~\ref{lem:noturn}, for every $\delta'>\delta>0$ there exist $\jjv_1<\jjv_2<\jjv_3\in\{0,\dots,N-1\}^{n+1}$, $l'\in B_{\delta}(F(\Sigma))$ and $x\in\R^d$
$$
\mathrm{cv}(z_{\jjv_1},z_{\jjv_2})\cap l'(x)\neq \emptyset\text{ and }\mathrm{cv}(z_{\jjv_2},z_{\jjv_3})\cap l'(x)\neq\emptyset.
$$
Since \eqref{eq:fzn} holds for every $k=0,\dots,N-1$, there are $k< m$ such that $z_{\jjv_1}\in f_k(Z_{n-1})$ and $z_{\jjv_3}\in f_m(Z_{n-1})$. On the other hand by \eqref{eq:fzn}, one of the endpoints of $f_k(Z_{n-1})$ (and $f_m(Z_{n-1})$) must be on the same side of $l'(x)$, where $z_{\jjv_1}$ (and $z_{\jjv_3}$ respectively) is. Denote these endpoints by $z_{a'}$ and $z_{b'}$. Observe that $z_{a'}\neq z_{b'}$. Indeed, if $z_{a'}=z_{b'}$ then $k=m-1$, and thus either $z_{\jjv_2}\in f_k(Z_{n-1})$ or $z_{\jjv_2}\in f_m(Z_{n-1})$. Hence, but $z_{\jjv_2}$ is separated from $z_{\jjv_1},z_{\jjv_3},z_{a'},z_{b'}$ by the plane $l'(x)$, which cannot happen by \eqref{eq:fzn}.

Moreover,by \eqref{eq:fzn}, one of the endpoints of
$f_{(\jjv_2)_0}(Z_{n-1})$ is on the same side of $l'(x)$ with
$z_{\jjv_2}$, denote it by $z_{c'}$. But the endpoints of
$f_p(Z_{n-1})$ are the elements of $Z_0$, moreover, $a'<c'<b'$,
which contradicts to the well ordered property of $Z_0$.
\end{proof}

\begin{theorem}\label{thm:equi}
    Let $\mathcal{S}$ be a non-degenerate system. Then the following three statements are equivalent
    \begin{enumerate}
        \item\label{t1} $S$ satisfies Assumption A,
        \item\label{t2} for $\mathcal{L}$-a.e. $x$, $\alpha_r(x)$ exists,
        \item\label{t3} $Z_0$ satisfies the well-ordered property.
    \end{enumerate}
\end{theorem}

\begin{proof}[Proof of Theorem~\ref{thm:equi}\eqref{t1}$\Rightarrow$Theorem~\ref{thm:equi}\eqref{t2}]
    Follows directly by the combination of Lemma~\ref{lem:erg} and Lemma~\ref{lem:ubassA}.
\end{proof}

\begin{proof}[Proof of Theorem~\ref{thm:equi}\eqref{t2}$\Rightarrow$Theorem~\ref{thm:equi}\eqref{t3}]
Let us argue by contradiction. Assume that $\alpha_r(x)$ exists for
$\mathcal{L}$-a.e. $x$ but there exists $Z_n$, $n\geq0$, which does
not satisfy the well-ordered property. By Lemma~\ref{lem:iff}, $Z_0$
does not satisfy the well ordered property. By
Lemma~\ref{lem:noturn}, let $l\in F(\Sigma)$, $x\in\R^d$ and
$z_{i-1},z_i,z_{i+1}\in Z_0$ be such that
$\mathrm{cv}(z_{i-1},z_{i})\cap l(x)\neq \emptyset$ and
$\mathrm{cv}(z_{i},z_{i+1})\cap l(x)\neq\emptyset$. By continuity of
the curve $\Gamma$, there exist $\ii,\jj\in\Sigma$ such that
$i_1=i-1\neq i=j_1$, $j_2\neq 0$ and $\Pi(\ii)-\Pi(\jj)\in l(x')$
with some $x'\in\R^d$. Hence,
$\langle\Pi(\ii)-\Pi(\jj)\rangle\subset l$. By definition, there
exists a $\mathbf{k}\in\Sigma$ such that $F(\mathbf{k})=l$. By using
the continuity of $F\colon\Sigma\mapsto G(d,d-1)$, $\Gamma$ and
$\Pi\colon\Sigma\mapsto\Gamma$, one can choose $n,m$ sufficiently
large, such that for every $\ii'\in[\ii|_n]$ and
$\mathbf{k}'\in[\mathbf{k}|_m]$,  \begin{equation}\label{eq:usethis}
F(\mathbf{k}')(\Pi(\ii'))\cap\Gamma_{j_1j_2}\neq\emptyset.
    \end{equation}

By ergodicity, for $\mu_0$-a.e. $\ii$, $\sigma^p\ii\in[k_m,\dots,k_1,i_1,\dots,i_n]$ for infinitely many $p\geq0$, where $\mathbf{k}|_m=k_1,\ldots,k_m$. Let us denote this subsequence by $p_k$.  Let $\mathbf{k}_k$ be the sequence such that $\mathbf{k}_k\in[k_1,\ldots,k_m,i_{p_k},\ldots,i_1]$.

By \eqref{eq:usethis}, there exists a sequence $\{\jj_{k}\}$ such that $\jj_k\wedge\ii=p_k+m$, $\sigma^{p_k+m}\jj_k\in[j_1j_2]$, and $\Pi(\sigma^{p_k+m}\jj_k)-\Pi(\sigma^{p_k+m}\ii)\in F(\mathbf{k}_k)$. By construction, $\sigma^{p_k+m}\jj_k\wedge\sigma^{p_k+m}\ii=0$ and $\sigma^{p_k+m}\jj_k\vee\sigma^{p_k+m}\ii=0$, and hence there exists a constant $c>0$ such that \newline$\|\Pi(\sigma^{p_k+m}\jj_k)-\Pi(\sigma^{p_k+m}\ii)\|>c$. Therefore for $\mu_0$-a.e. $\ii$
\begin{align*}
  \alpha_r(\pi(\ii))&=\lim_{\pi(\jj)\to\pi(\ii)}\frac{\log\|\Pi(\ii)-\Pi(\jj)\|}{\log|\pi(\ii)-\pi(\jj)|}=\lim_{k\to+\infty}\frac{\log\|\Pi(\ii)-\Pi(\jj_k)\|}{\log|\pi(\ii)-\pi(\jj_k)|}\\
  &=\lim_{k\to+\infty}\frac{\log\|A_{\ii|_{p_k+m}}(\Pi(\sigma^{p_k+m}\ii)-\Pi(\sigma^{p_k+m}\jj_k))\|}{\log|\lambda_{\ii|_{p_k+m}}(\pi(\sigma^{p_k+m}\ii)-\pi(\sigma^{p_k+m}\jj_k))|}\\
  &\geq\lim_{k\to+\infty}\frac{\log\|A_{\ii|_{p_k+m}}|F(\mathbf{k}_k)\|}{\log\lambda_{\ii|_{p_k+m}}}\geq\lim_{k\to+\infty}\frac{\log\alpha_2(A_{\ii|_{p_k+m}})}{\log\lambda_{\ii|_{p_k+m}}}\\
  &\geq\frac{\log\tau}{-\chi_{\mu_0}}+\lim_{k\to+\infty}\frac{\log\|A_{\ii|_{p_k+m}}\|}{\log\lambda_{\ii|_{p_k+m}}}=\frac{\log\tau}{-\chi_{\mu_0}}+\alpha(\pi(\ii)),
\end{align*}
\eqref{eq:GibbsLyap} where
Theorem~\ref{thm:BochiGour}\eqref{item:bg4},
Theorem~\ref{thm:BochiGour}\eqref{item:bg6} and Lemma~\ref{lem:erg}.
But $-\log\tau/\chi_{\mu_0}>0$, which is a contradiction.
\end{proof}

Let us recall that for any $\underline{0}\neq x\in\R^d$, $\langle v\rangle$ denotes the unique $1$-dimensional subspace in $\mathbb{PR}^{d-1}$ such that $v\in\langle v\rangle$. Also, any $V\in G(d,d-1)$ can be identified with a $d-2$ dimensional, closed submanifold $\tilde{V}$ of $\mathbb{PR}^{d-1}$ such that $\tilde{V}=\{\theta\in\mathbb{PR}^{d-1}:\theta\subset V\}$. Also, for a subset $B\subset G(d,d-1)$ we can identify it with a subset $\tilde{B}$ of $\mathbb{PR}^{d-1}$ such that $\tilde{B}=\{\theta\in\mathbb{PR}^{d-1}:\theta\subset V\in B\}$.

\begin{proof}[Proof of Theorem~\ref{thm:equi}\eqref{t3}$\Rightarrow$Theorem~\ref{thm:equi}\eqref{t1}]
Suppose that $Z_0$ satisfies the well ordered property. By
Lemma~\ref{lem:iff}, $Z_n$ satisfies the well-ordered property for
every $n\geq0$ and thus, we may assume that $\langle
z_{\iiv}-z_{\jjv}\rangle\notin\widetilde{F(\Sigma)}$ for every
$z_{\iiv},z_{\jjv}\in Z_n$. Indeed, if $\langle
z_{\iiv}-z_{\jjv}\rangle\in\widetilde{F(\ii)}$ for some
$\ii\in\Sigma$ then one could find
$z_{\iiv_1'},z_{\iiv_2'},z_{\iiv_3'}\in Z_{n+1}$ such that
$\iiv_1'<\iiv_2'<\iiv_3'$, $z_{\iiv_1'}=z_{\iiv}$,
$z_{\iiv_3'}=z_{\jjv}$ and
$$\proj_{F(\ii)^{\perp}}(z_{\iiv_2'})\notin\mathrm{cv}(\proj_{F(\ii)^{\perp}}(z_{\iiv_1'}),\proj_{F(\ii)^{\perp}}(z_{\iiv_3'})).$$

So, for every $\iiv,\jjv\in\Sigma^*$ there exists open, connected component $C_{\iiv,\jjv}$  of $\mathbb{PR}^{d-1}\setminus\widetilde{F(\Sigma)}$ such that $\langle z_{\iiv}-z_{\jjv}\rangle\in C_{\iiv,\jjv}$.

Then for any $\iiv_1<\iiv_2<\iiv_3$, $C_{\iiv_1,\iiv_2}=C_{\iiv_1,\iiv_3}=C_{\iiv_2,\iiv_3}$. Indeed, if $C_{\iiv_1,\iiv_2}\neq C_{\iiv_1,\iiv_3}$ then there exists $\widetilde{F(\ii)}$, which separates $\langle z_{\iiv_2}-z_{\iiv_1}\rangle$ and $\langle z_{\iiv_3}-z_{\iiv_1}\rangle$. But then, for $F(\ii)^{\perp}$, $\proj_{F(\ii)^{\perp}}(z_{\iiv_2})\notin\mathrm{cv}(\proj_{F(\ii)^{\perp}}(z_{\iiv_1}),\proj_{F(\ii)^{\perp}}(z_{\iiv_3}))$, which cannot happen by definition of well ordered property.

Therefore, there exists a unique open,connected component $C$ such that $\langle z_{\iiv}-z_{\jjv}\rangle\in C$ for every $\iiv,\jjv\in\Sigma^*$. But, for any $\ii\in\Sigma$, since $\langle z_N-z_0\rangle\notin\widetilde{F(\Sigma)}$
$$
\lim_{n\to+\infty}\langle A_{\ii|_n}(z_N-z_0)\rangle= E(\ii),
$$
hence, $E(\Sigma)\subset C$. Thus, for any multicone $M$, for which the dominated splitting condition of index-1 holds, the cone $M\cap C$ is invariant, i.e. $A_i (M\cap C)\subset M^o\cap C$.

On the other hand, by $\langle z_N-z_0\rangle\in C$, one can extend $M\cap C$ such that $\langle z_N-z_0\rangle\in M\cap C$ and $M\cap C$ remains invariant.
\end{proof}

\section{An example, de Rham's curve}\label{Sec:Ex_deRham}

In this last section of the paper, we show an application for our
main theorems. The well-known de Rham's curve in $\R^2$ is the
attractor of the affine zipper, formed by the functions
\begin{equation}\label{eq:defdeRham}
f_0(x)=\begin{bmatrix}
\omega & 0 \\
\omega & 1-2\omega \\
\end{bmatrix} x-\begin{bmatrix} 0 \\ 2\omega \\ \end{bmatrix}
\;\;\text{ and }\;\;
f_1(x)=\begin{bmatrix}
1-2\omega & \omega \\
0 & \omega \\
\end{bmatrix} x+\begin{bmatrix} 2\omega \\ 0 \\ \end{bmatrix},
\end{equation}
where $\omega\in(0,1/2)$ is a parameter.

Originally, the curve was introduced and studied by de Rham
\cite{deRham1, deRham2, deRham3} with a geometric construction.
Starting from a square, it can be obtained by trisecting each side
with ratios $\omega:(1-2\omega):\omega$ and "cutting the corners" by
connecting each adjacent partitioning point to get an octagon.
Again, each side is divided into three parts with the same ratio and
adjacent partitioning points are connected, and so on. The de Rham
curve is the limit curve of this procedure. More precisely, the
curve defined by the zipper in \eqref{eq:defdeRham} gives the
segment between two midpoints of the original square.

Let us define the following linear parametrisation of the curve. Let $v\colon[0,1]\mapsto\R^2$ the function of the form
\begin{equation}\label{eq:defdeRham2}
v(x)=f_i(v(2x-i)), \text{ for } x\in \Big[\frac{i}{2}, \frac{i+1}{2}\Big),\; i=0,1.
\end{equation}

For a visualisation of a linearly parametrized de Rham curve, 
see Figure~\ref{fig:derham}. 

\begin{figure}[!htb]
\minipage{0.32\textwidth}
  \includegraphics[width=\linewidth]{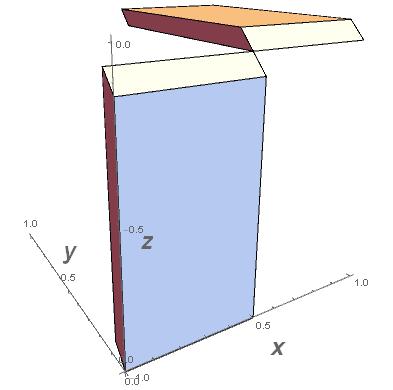}
\endminipage\hfill
\minipage{0.32\textwidth}
  \includegraphics[width=\linewidth]{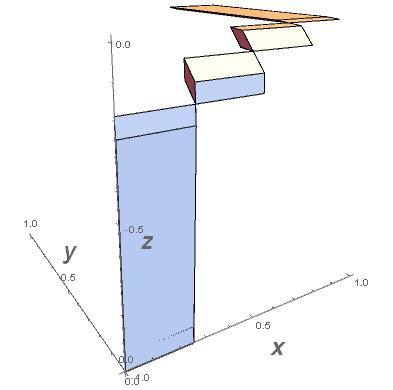}
\endminipage\hfill
\minipage{0.32\textwidth}%
  \includegraphics[width=\linewidth]{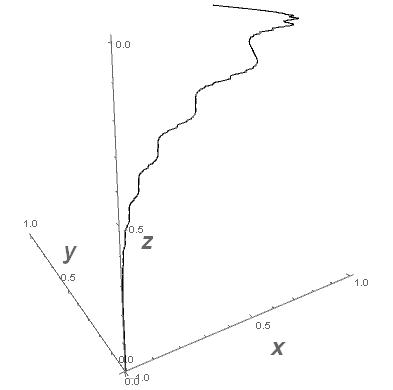}
\endminipage
\caption{Linearly parametrized de Rham curve with parameter $\omega=1/10$. Left: The image of the unit cube w.r.t the IFS generating the graph of the de Rham curve. Middle: The second iteration. Right: The curve itself.}\label{fig:derham}
\end{figure}

Protasov \cite{Protasov04, Protasov06} proved in a more general context that the set of points $x\in[0,1]$ for which $\alpha(x)=\beta$ has full measure only if $\beta=\widehat{\alpha}$, otherwise it has zero measure. Just recently, Okamura \cite{Okamura} bounds $\alpha(x)$ for Lebesgue typical points allowing in the definition \eqref{eq:defdeRham2} more than two functions and also non-linear functions under some conditions.

We show that with a suitable coordinate transform the matrices $A_0$ and $A_1$ satisfy Assumption A and hence, our results are applicable.

\begin{lemma}\label{lem:pozdeRham}
    For every $\omega\in(0,1/3)\cup(1/3,1/2)$ there exists a coordinate transform $D(\omega)$ such that $D(\omega)^{-1}A_iD(\omega)$ has strictly positive entries for $i=1,2$.
\end{lemma}

\begin{proof}
    For $\varepsilon>0$ and $0<\delta<1$ define the coordinate transform matrices
    \[ \widetilde{D}_{\varepsilon}=\begin{bmatrix} 1 & \varepsilon \\ \varepsilon & 1 \\ \end{bmatrix} \;\text{ and }\; \widehat{D}_{\delta}=\begin{bmatrix} 1 & -\delta \\ -\delta & 1 \\ \end{bmatrix}. \]
    Elementary calculations show that the matrices $\widetilde{A}_0=\widetilde{D}_\varepsilon^{-1}A_0\widetilde{D}_\varepsilon$ and $\widetilde{A}_1=\widetilde{D}_\varepsilon^{-1}A_1\widetilde{D}_\varepsilon$ have strictly positive entries whenever
    \[ \frac{1}{3-\varepsilon}<\omega<\frac{1}{2-\varepsilon-\varepsilon^2}. \]
    The largest possible interval $(1/3,1/2)$ is attained when $\varepsilon$ is arbitrarily small. Very similar calculations show that the entries of $\widehat{A}_0=\widehat{D}_\delta^{-1}A_0\widehat{D}_\delta$ and $\widehat{A}_1=\widehat{D}_\delta^{-1}A_1\widehat{D}_\delta$ are strictly positive whenever
    \[ \frac{\delta}{1+3\delta}<\omega<\frac{1}{3+\delta}, \]
    which gives the open interval $(0,1/3)$. Also trivial calculations show that $\|\widetilde{A}_i\|_1=\|A_i\|_1=\|\widehat{A}_i\|_1$, $i=0,1$.
\end{proof}

Let us recall that in this case $P(t)$ has the form
$$
P(t)=\lim_{n\to+\infty}\frac{-1}{n\log2}\log\sum_{|\iiv|=n}\|A_{\iiv}\|^t,
$$
and
$$
\alpha_{\min}=\lim_{t\to+\infty}\frac{P(t)}{t}\text{ and }\alpha_{\max}=\lim_{t\to-\infty}\frac{P(t)}{t}.
$$

\begin{prop}
    For every $\omega\in(0,1/4)\cup(1/4,1/3)\cup(1/3,1/2)$ the de Rham function $v:[0,1]\mapsto\R^2$, defined in \eqref{eq:defdeRham2}, the following are true
    \begin{enumerate}
        \item\label{item:dr1} $v$ is differentiable for Lebesgue-almost every $x\in[0,1]$ with derivative vector equal to zero,
        \item\label{item:dr2} Let $\mathcal{N}$ be the set of $[0,1]$ such that $v$ is not differentiable. Then $\dim_H\mathcal{N}=\tau-P(\tau)>0$, where $\tau\in\R$ is chosen such that $P'(\tau)=1$,
        \item\label{item:dr3} $\dim_H\Pi_*\mu_0<1$, where $\mu_0=\left\{\frac{1}{2},\frac{1}{2}\right\}^{\mathbb{N}}$ equidistributed measure on $\Sigma=\left\{0,1\right\}^{\mathbb{N}}$ and $\Pi$ is the natural projection from $\Sigma$ to $v([0,1])$.
        \item\label{item:dr4} for every $\beta\in[\alpha_{\min},\alpha_{\max}]$
        \begin{align*}
        \dim_H\left\{x\in[0,1]:\alpha(x)=\beta\right\}&=\dim_H\left\{x\in[0,1]:\alpha_r(x)=\beta\right\}\\
        &=\inf_{t\in\R}\{t\beta-P(t)\}.
        \end{align*}
    \end{enumerate}
\end{prop}

For $\omega=1/4$ the de Rham curve is a smooth curve, namely a parabola arc. For $\omega=1/3$, the matrices does not satisfy the dominated splitting condition. For this case, we refer to the work of Nikitin \cite{Nikitin}.

\begin{proof}
By Lemma~\ref{lem:pozdeRham}, we are able to apply
Theorem~\ref{thm:assA} and Theorem~\ref{thm:equiv} for
$\omega\neq1/3$. It is easy to see that \eqref{eq:defdeRham} is
symmetric. Thus, by Theorem~\ref{thm:assA}, the statement
\eqref{item:dr4} of the proposition follows.

On the other hand, let $\mathcal{N}$ be the set, where $v$ is not differentiable. Then
$$
\{x\in[0,1]:\alpha(x)<1\}\subseteq\mathcal{N}\subseteq\{x\in[0,1]:\alpha(x)\leq1\}.
$$
Thus, $\dim_H\mathcal{N}=\inf_{t\in\R}\{t-P(t)\}$.

    Now, we prove that there exists $\tau\in\R$ such that $P'(\tau)=1$. Observe that $M=A_0+A_1$ is a stochastic matrix with left and right eigenvectors $p=(p_1,p_2)^T$ and $e=(1,1)^T$, respectively,
    corresponding to eigenvalue $1$ and $p_i>0$, $p_1+p_2=1$. There exists a constant $c>0$ such that for every $\iiv\in\Sigma^*$
    $$
    c^{-1}p^TA_{\iiv}e\leq\|A_{\iiv}\|\leq cp^TA_{\iiv}e,
    $$
    and therefore
    $$
    \sum_{|\iiv|=n}p^TA_{\iiv}e=p^T(A_0+A_1)^ne.
    $$
    Thus $P(1)=0$, and
    \begin{eqnarray*}
    &&\mu_1([{\bf i}|_n])=p^T A_{{\bf i}|_n} e= p^TA_{i_1}\ldots A_{i_n}e \text{ and }\\
    &&\mu_0([{\bf i}|_n])=\frac{1}{2^n}\text{ for every } {\bf i}\in \Sigma.
    \end{eqnarray*}

    Simple calculations show that,
    \[
    \mu_1([00])=(\frac{1}{2},\frac{1}{2})A_0A_0(1,1)^T=\omega^2+\frac{(1-2\omega)\omega}{2}+\frac{(1-2\omega)^2}{2},
    \]
    which is not equal to $1/4$ if $\omega\neq1/4$ or $\omega\neq1/2$. Thus, for $\omega\neq1/4$ and $\omega\neq1/2$, $\mu_0\neq\mu_1$, and by Lemma~\ref{lem:boundsonder}, $P'(1)<1<P'(0)$. Since $t\mapsto P'(t)$ is continuous, there exists $\tau$ such that $P'(\tau)=1$ and therefore  $\dim_H\mathcal{N}=\tau-P(\tau)>0$, which completes \eqref{item:dr2}.

    On the other hand, by Theorem~\ref{thm:equiv}, $\alpha_r(x)=P'(0)>1$ for Lebesgue almost every $x\in[0,1]$ and therefore $v$ is differentiable with derivative vector $\underline{0}$. This implies \eqref{item:dr1}.

    Finally, we show statement \eqref{item:dr3} of the proposition. By using the classical result of Young
    $$\dim_H\Pi_*\mu_0=\liminf_{r\to 0+}\frac{\log\Pi_*\mu_0(B_r(x))}{\log r}\text{ for }\pi_*\mu_0\text{-a.e. }x\in\Gamma=v([0,1]).$$

    For an $\ii\in\Sigma$ and $r\in\R$ let $n\geq1$ be such that $\|A_{{\bf i}|_n}\|\leq r<\|A_{{\bf i}|_{n-1}}\|$. Hence, $\Pi([\ii|_n])\subseteq B_r(\Pi(\ii))$ and
    \[
    \liminf_{r\to 0+}\frac{\log\Pi_*\mu_0(B_r(\Pi(\ii)))}{\log r}\leq\liminf_{n\to \infty}\frac{\log\mathbb{P}([{\bf i}|_n])}{\log \|A_{{\bf i}|_n}\|}
    \]
    Since $\mu_0([{\bf i}|n])=1/2^n$, by Proposition~\ref{prop:pressure}
    $$\liminf_{n\to \infty}\frac{\log\mathbb{P}([{\bf i}|_n])}{\log \|A_{{\bf i}|_n}\|}=\frac{1}{P'(0)}<1.$$
    \end{proof}

\begin{remark}\label{rem:final}
    Finally, we remark that in case of general signature vector, one may modify the definition of $\ii\vee\jj$ to
\begin{equation*} 
\ii\vee\jj =
\left\{
\begin{array}{ll}
\min \{\sigma^{\ii\wedge\jj+1}\ii\wedge\1,\sigma^{\ii\wedge\jj+1}\jj\wedge\0\}, & \hbox{$i_{\ii\wedge\jj+1}+1=j_{\ii\wedge\jj+1}\ \&\ \varepsilon_{i_{\ii\wedge\jj+1}}=0\ \&\ \varepsilon_{j_{\ii\wedge\jj+1}}=0$,} \\
\min\{\sigma^{\ii\wedge\jj+1}\ii\wedge\0,\sigma^{\ii\wedge\jj+1}\jj\wedge\0\}, & \hbox{$i_{\ii\wedge\jj+1}+1=j_{\ii\wedge\jj+1}\ \&\ \varepsilon_{i_{\ii\wedge\jj+1}}=1\ \&\ \varepsilon_{j_{\ii\wedge\jj+1}}=0$,} \\
\min\{\sigma^{\ii\wedge\jj+1}\ii\wedge\0,\sigma^{\ii\wedge\jj+1}\jj\wedge\1\}, & \hbox{$i_{\ii\wedge\jj+1}+1=j_{\ii\wedge\jj+1}\ \&\ \varepsilon_{i_{\ii\wedge\jj+1}}=1\ \&\ \varepsilon_{j_{\ii\wedge\jj+1}}=1$,} \\
\min\{\sigma^{\ii\wedge\jj+1}\ii\wedge\1,\sigma^{\ii\wedge\jj+1}\jj\wedge\1\}, & \hbox{$i_{\ii\wedge\jj+1}+1=j_{\ii\wedge\jj+1}\ \&\ \varepsilon_{i_{\ii\wedge\jj+1}}=0\ \&\ \varepsilon_{j_{\ii\wedge\jj+1}}=1$,} \\
\min \{\sigma^{\ii\wedge\jj+1}\jj\wedge\1,\sigma^{\ii\wedge\jj+1}\ii\wedge\0\}, & \hbox{$j_{\ii\wedge\jj+1}+1=i_{\ii\wedge\jj+1}\ \&\ \varepsilon_{j_{\ii\wedge\jj+1}}=0\ \&\ \varepsilon_{i_{\ii\wedge\jj+1}}=0$,} \\
\min\{\sigma^{\ii\wedge\jj+1}\jj\wedge\0,\sigma^{\ii\wedge\jj+1}\ii\wedge\0\}, & \hbox{$j_{\ii\wedge\jj+1}+1=i_{\ii\wedge\jj+1}\ \&\ \varepsilon_{j_{\ii\wedge\jj+1}}=1\ \&\ \varepsilon_{i_{\ii\wedge\jj+1}}=0$,} \\
\min\{\sigma^{\ii\wedge\jj+1}\jj\wedge\0,\sigma^{\ii\wedge\jj+1}\ii\wedge\1\}, & \hbox{$j_{\ii\wedge\jj+1}+1=i_{\ii\wedge\jj+1}\ \&\ \varepsilon_{j_{\ii\wedge\jj+1}}=1\ \&\ \varepsilon_{i_{\ii\wedge\jj+1}}=1$,} \\
\min\{\sigma^{\ii\wedge\jj+1}\jj\wedge\1,\sigma^{\ii\wedge\jj+1}\ii\wedge\1\}, & \hbox{$j_{\ii\wedge\jj+1}+1=i_{\ii\wedge\jj+1}\ \&\ \varepsilon_{j_{\ii\wedge\jj+1}}=0\ \&\ \varepsilon_{i_{\ii\wedge\jj+1}}=1$,} \\
0, & \hbox{otherwise.}
\end{array}
\right.
\end{equation*}
\end{remark}

\hspace{3mm}
\noindent{\bf Acknowledgement.} B\'ar\'any was partially supported by
the grant OTKA K104745 and the grant EPSRC EP/J013560/1. Kiss was partially supported by the grant OTKA K104178 and the internal research project F1R-MTH-PUL-15MRO3 of the University of Luxembourg. Kolossv\'ary was partially supported by the Hungarian National E\"otv\"os Scholarship.

\bibliographystyle{abbrv}
\bibliography{deRhambiblio}

\end{document}